\documentclass[a4 paper]{article}

\usepackage{moreverb}

\usepackage[colorlinks,bookmarksopen,bookmarksnumbered,
             citecolor=red,urlcolor=red]{hyperref}

\usepackage{graphicx}        
\usepackage{multicol}        
\usepackage[bottom]{footmisc}
\usepackage{bm}
\usepackage{amsmath}
\usepackage{amssymb}
\usepackage{caption}
\usepackage[noadjust]{cite}

 \newtheorem{thm}{Theorem}
 \newtheorem{cor}[thm]{Corollary}
 \newtheorem{lem}[thm]{Lemma}
 \newtheorem{prop}[thm]{Proposition}
 \newtheorem{defn}[thm]{Definition}
 \newtheorem{rmk}[thm]{Remark}
 
 \newenvironment{proof}{\par\noindent{\bf Proof.}}{\par\rightline{$\blacksquare$}}

       \newcommand{\be}{\begin{equation}}
       \newcommand{\ee}{\end{equation}}

\newcommand\BibTeX{{\rmfamily B\kern-.05em \textsc{i\kern-.025em b}\kern-.08em
T\kern-.1667em\lower.7ex\hbox{E}\kern-.125emX}}

\begin{document}

\author{R. Ghomrasni and R. \L{}ochowski}

\title{The Play Operator, the Truncated Variation and the Generalisation
of the Jordan Decomposition}

\author{R. Ghomrasni and R. \L{}ochowski}



\maketitle

\vspace{-6pt}

\begin{abstract}
The play operator minimalizes the total variation on intervals $[0,T], T> 0$,
of functions approximating uniformly given regulated function with given accuracy and starting from a given point. In this article we link the play operator with so called truncated variation functionals, introduced recently by the second-named author, and provide a semi-explicit expression for the play operator in terms of these functionals. Generalisation for time-dependent boundaries is also considered. This gives the best possible lower bounds for the total variation of the outputs of the play operator and its Jordan-like decomposition.
\end{abstract}

\section{Introduction}

The so called play operator is an important non-linear operator encountered in mathematical models of hysteresis \cite{BerM},
\cite{KP89}. It also appears in the models on optimal hedging of
options with transaction costs \cite{WW97}, \cite{WW99}, \cite{Za06}.
The intuition leading to the play operator may be stated as follows:
for a given input function we look for an output function, starting
value of which is given and we do no change its value as long as the
difference between input and output stays within a given boundary
region. Several definitions of the play operator are possible, depending
on the regularity of the input, the space in which the input and output
attain their values and the boundary conditions (characteristics)
\cite{KL2}, \cite{R2007}, \cite{ChK92}, \cite{IE}. It may be defined for sufficiently regular inputs and boundary conditions (characteristics) e.g. with
the Young or Kurzweil integral formalism \cite{KL2}, \cite{BK}.
In \cite{BK} one finds the following definition of the play operator.
Let $X$ be a Hilbert space with inner product $\left\langle \cdot,\cdot\right\rangle $
and $Z$ a (not necessarily bounded) convex, closed subset
of $X,$ containing a ball with the center at $0$ and with radius
$r>0.$ For a given initial condition $x_{0}\in Z$ and an input function
$u:\left[a;b\right]\rightarrow X,$ which is c\`{a}gl\`{a}d, i.e. left-continuous
with right limits, we look for a c\`{a}gl\`{a}d, bounded variation solution
$\xi:\left[a;b\right]\rightarrow X$ of the following problem $({\cal P}):$
\begin{enumerate}
\item $u\left(t\right)-\xi\left(t\right)\in Z$ for every $t\in\left[a;b\right];$
\item $u\left(a\right)-\xi\left(a\right)=x_{0};$
\item $(K)\int_{a}^{b}\left\langle u\left(t+\right)-\xi\left(t+\right)-y\left(t\right),\mathrm{d}\xi\left(t\right)\right\rangle \geq0$
for any regulated function $y:\left[a;b\right]\rightarrow Z.$
\end{enumerate}
The integral $(K)\int$ is understood as the Kurzweil integral and a
regulated function is a function with right and left limits. The problem
$({\cal P})$ admits a unique solution $\xi,$ which is interpreted
as the output of the play operator with the input $u.$

The play operator may be also defined for time-dependent characteristics
(which we will call sometimes boundary conditions or boundaries) \cite{ChK92},
\cite{K1} or for more general inputs, like $L^{\infty}\left(\left[a;b\right]\right)$
or even Lebesgue measurable functions \cite{KL1}, \cite{IE}, where
the Kurzweil integral formalism can not be applied in a straightforward
way. 

In this article we will restrict to one-dimensional regulated inputs
and regulated, time-dependent characteristics. Possible applications are encountered in real
industry problems. 
For example,
in \cite{Za06} we find the following description of a hedging strategy
for a European call option with transaction costs: ``the numerical
calculations show that the optimal hedge ratio $\Delta$ is constrained
to evolve between two boundaries, $\Delta_{l}$ and $\Delta_{u},$
such that $\Delta_{l}<\Delta_{u}.$ As long as the hedge lies within
these two boundaries, $\Delta_{l}\leq\Delta\leq\Delta_{u},$ no rebalancing
of the hedging portfolio takes place. That is why the region between
the two boundaries is commonly denoted as the 'no transaction region'.
As soon as the hedge ratio goes out of the no transaction region,
a rebalancing occurs in order to bring the hedge to the nearest boundary
of the no transaction region.'' In the above example the input is the optimal hedge ratio, which, similar as boundaries, is a c\`{a}dl\`{a}g process. This stems from the fact that these processes shall be "non-anticipating", reflecting the fundamental assumption of financial mathematics models that one can not predict with full certainty future movements of the market.

In this article we link the play operator with functionals called
$\alpha,\beta-$truncated variation and truncated variation (denoted $TV^c,$ with $c>0$); the latter is the special case of the former, with $\alpha \equiv -c/2$ and $\beta \equiv c/2$. This
link stems from the so called laziness principle of the play operator.
Truncated variation functionals give the greatest lower bound for the
total variation of any function $\xi$ such that the difference between
$\xi$ and a given function $u$ lies between $\alpha$ and $\beta,$
and may be applied for general regulated inputs and characteristics
$\alpha\leq\beta$ such that the output has locally bounded variation.
It seems that $\alpha,\beta-$truncated variation formalism lies in
between the Kurzweil (cf. \cite{K1}) or the Lebesgue-Stieltjes integral
formalism (cf. \cite{R2008}) and the standard approximation arguments,
for more general inputs and boundaries. This does not extend much
the definition of the play operator beyond the results of \cite[Corollary 2.3]{K1} but it gives better insight when the output has locally bounded variation. Moreover, it gives (the best possible) lower bounds of the total variation of the output and its Jordan-like decomposition.
 
In \cite{L1}, the best possible lower bound  (\ref{eq:inf_tv1}) of the total variation of the output of the play operator for a c\`{a}dl\`{a}g (right-continuous with left limits) input and constant, symmetric boundaries was given. For every $c,t>0$ and every regulated input $u,$ the output of the
play operator with constant, symmetric boundaries $-c/2,$ $c/2$
has locally finite total variation which, due to estimates (\ref{eq:inf_tv1})
and (\ref{eq:sup_tv1}), is comparable for small $c$s with $TV^{c}\left(u,\left[0;t\right]\right).$
On the other hand, as $c\downarrow0,$ its variation tends to the
total variation of the input. The natural question arises, what
is the rate of the convergence of $TV^{c}\left(u,\left[0;t\right]\right)$
to $+\infty,$ when $TV\left(u,\left[0;t\right]\right)=+\infty.$
A partial answer to this question was given by \cite[Theorem 17]{TV00}.
In \cite{TV00}, for $p\geq1$ and $T>0$ the following classes are
defined: $\mathcal{V}_{p}$ is the class of functions $\psi:\left[0;T\right]\rightarrow\mathbb{R}$
with finite $p-$variation 
\begin{equation}
V^{p}\left(\psi,\left[0;T\right]\right):=\sup_{n}\sup_{0\leq t_{0}<t_{1}<\cdots<t_{n}\leq T}\sum_{i=1}^{n}\left|\psi\left(t_{i}\right)-\psi\left(t_{i-1}\right)\right|^{p}\label{eq:p_var}
\end{equation}
and $\mathcal{U}_{p}$ is the class of such functions $\psi$ that
$\limsup_{c\downarrow0}c^{p-1}TV^{c}\left(\psi,\left[0;T\right]\right)$
is greater than $0$ but finite. In \cite{TV00} it was shown that
for any $p\geq1$ and $\delta>0$ we have inclusions $\mathcal{V}_{p}\subset U_{p}\subset\mathcal{V}_{p+\delta}$
and for $p>1$ these inclusions are strict. 
In \cite{TV00} it was also mentioned, that due to the fact that typical
path of a standard Brownian motion $B_{t},$ $t\geq0,$ has, with probability
$1,$ finite $2-$variation, the variations $TV^{c}\left(B,\left[0;T\right]\right)$
have the growth rate $1/c$ as $c\downarrow0$ with probability $1.$
It is not true that with probability $1,$ $B\in\mathcal{V}_{2}$
when $p-$variation is defined with formula (\ref{eq:p_var}), cf.
\cite{Levy}, but the conclusion remains true. In \cite[Theorem 1]{LM13}
more general fact was shown - that for any continuous semimartingale
$X_{t},$ $t\geq0,$ the function $T\mapsto c\cdot TV^{c}\left(X,\left[0;T\right]\right),$
$T\geq0,$ almost surely converges uniformly on compact subsets of
$\left[0;+\infty\right)$ to the quadratic variation $\left\langle X\right\rangle _{T}$ of $X,$ 
which, in the case of a standard Brownian motion equals $T.$ More
accurate results were obtained for diffusions (cf. \cite[Theorem 5]{LM13}). 

A very interesting tree-approach of \cite{P08} allows to obtain estimates
of the total variation of the output of the play operator for inputs
being typical paths of fractional Brownian motions and L\'evy processes.
Fractional Brownian motion with Hurst parameter $H\in\left(0;1\right)$
is a self-similar, continuous, Gaussian process $B_{t}^{H},$ $t\geq0,$
such that for $\lambda>0,$ $\lambda^{-H}B_{\lambda t}^{H},$ $t\geq0,$
has the same probability distribution as $B^{H}.$ The easy consequence
of \cite[Proposition 3.7]{P08} is that the variations $TV^{c}\left(B^{H},\left[0;T\right]\right)$
have, with probability $1,$ the growth rate $c^{1-1/H}$ as $c\downarrow0.$ The quantity $L^a$ appearing in \cite[Proposition 3.7]{P08} differs from $\frac{1}{2}TV^{a}\left(B^{H},\left[0;T\right]\right)$ at most by $\sup_{t\in[0;T]}B^{H}_t - \inf_{t\in[0;T]}B^{H}_t$ which is finite with probability $1.$
The case of L\'evy processes is studied in \cite[Proposition 3.14]{P08}
and this gives e.g. estimates of $TV^{c}$ for $\alpha-$stable processes.
For $\alpha\in\left(0;2\right],$ $\alpha-$stable L\'evy process is
a self-similar, c\`{a}dl\`{a}g process $X_{t}^{\alpha},$ $t\geq0,$ such that
for $\lambda>0,$ $\lambda^{-1/\alpha}X_{\lambda t}^{\alpha},$ $t\geq0,$
has the same probability distribution as $X_{t}^{\alpha},$ $t\geq0.$
The consequence of \cite[Proposition 3.14]{P08} is that the variations
$TV^{c}\left(X^{\alpha},\left[0;T\right]\right)$ have for $\alpha > 1$,
with probability $1,$ the growth rate $c^{1-\alpha}$ as $c\downarrow 0$ (see \cite[Formula (3.3)]{P08} and the discussion after the proof of \cite[Proposition 3.14]{P08}).
The mentioned results for self-similar processes may be also obtained
via the ergodic theorem. 

We also have the following simple general observation allowing to assess
the rate of convergence of $TV^{c}\left(\psi_{1}+\psi_{2},\left[0;T\right]\right)$
as $c\downarrow0,$ knowing the rates of convergence of $TV^{c}\left(\psi_{1},\left[0;T\right]\right)$
and $TV^{c}\left(\psi_{2},\left[0;T\right]\right).$ More precisely,
we have 
\begin{lem} \label{lema1}
Let $\psi_{1},\psi_{2}:\left[0;+\infty\right)\rightarrow\mathbb{R}.$
If there exists a non-decreasing, regularly varying function at
$0,$ $\varphi,$ such that 
\[
\lim_{c\downarrow0}\varphi\left(c\right)\cdot TV^{c}\left(\psi_{1},\left[0;t\right]\right)>0\mbox{ but }\limsup_{c\downarrow0}\varphi\left(c\right)\cdot TV^{c}\left(\psi_{2},\left[0;t\right]\right)=0
\]
then we also have 
\[
\lim_{c\downarrow0}\varphi\left(c\right)\cdot TV^{c}\left(\psi_{1}+\psi_{2},\left[0;t\right]\right)=\lim_{c\downarrow0}\varphi\left(c\right)\cdot TV^{c}\left(\psi_{1},\left[0;t\right]\right).
\]
\end{lem}
The proof of Lemma \ref{lema1} is given in the Appendix. 

Let us comment on the organisation of the paper. In the next section
we define play operator for any regulated input and constant, symmetric
boundaries, and present its relation with truncated variation. Simultaneously to presenting the play operator, we also report on
a very similar notion of the Skorohod map and its extensions, acting on the set of c\`{a}dl\`{a}g functions. In
the third section we extend the definition of the play operator with constant, symmetric boundaries to the case with time-dependent
boundaries and finite variation output, and relate it to (defined there) $\alpha,\beta-$ truncated variation. This relation is the subject of the main result of this paper - Theorem \ref{cond_tv_alpha_beta}. The last section is devoted to the definition of the play operator for any regulated function and regulated characteristics. In the Appendix we present the proof of  Lemma \ref{lema1}.

\section{The play operator with constant, symmetric
boundaries and the Skorohod problem}

As it was mentioned in the introduction, there exist several definitions
of the play operator, depending on the regularity of the input, the
space in which the input and output attain their values, and its characteristics.
In this section we will consider the play operator acting on one-dimensional
regulated inputs with constant, symmetric characteristics (boundaries)
$-c/2,$ $c/2$ where $c>0.$ 

We start with necessary definitions and notation. Let $\psi:[0;+\infty)\rightarrow\mathbb{R}.$
By $TV\left(\psi,\left[a;b\right]\right)$ we denote the \emph{total
variation} of the function $\psi$ on the interval $\left[a;b\right],$
$0\leq a<b<+\infty,$ i.e. 
\begin{equation}
TV\left(\psi,\left[a;b\right]\right):=\sup_{n}\sup_{a\leq t_{0}<t_{1}<\cdots<t_{n}\leq b}\sum_{i=1}^{n}\left|\psi\left(t_{i}\right)-\psi\left(t_{i-1}\right)\right|.\label{eq:tv}
\end{equation}
We also define \emph{positive and negative variations} of the function
$\psi$ on the interval $\left[a;b\right],$ $UTV\left(\psi,\left[a;b\right]\right)$
and $DTV\left(\psi,\left[a;b\right]\right)$ respectively, with the
following formulas 
\[
UTV\left(\psi,\left[a;b\right]\right):=\sup_{n}\sup_{a\leq t_{0}<t_{1}<\cdots<t_{n}\leq b}\sum_{i=1}^{n}\left(\psi\left(t_{i}\right)-\psi\left(t_{i-1}\right)\right)_{+},
\]
\[
DTV\left(\psi,\left[a;b\right]\right):=\sup_{n}\sup_{a\leq t_{0}<t_{1}<\cdots<t_{n}\leq b}\sum_{i=1}^{n}\left(\psi\left(t_{i-1}\right)-\psi\left(t_{i}\right)\right)_{+},
\]
where $\left(x\right)_{+}=\max\left\{ x,0\right\} .$

If $TV\left(\psi,\left[a;b\right]\right)<+\infty$ for any $0\leq a<b<+\infty$
we will call $\psi$ a {\em function with locally bounded variation}. For
such $\psi$ we have that $\psi$ is a \emph{regulated function},
i.e. it has left and right limits (with the convention that the left
limit at $0,$ $\psi\left(0-\right),$ equals $\psi\left(0\right)$),
and we have the following \emph{Jordan decomposition} 
\[
TV\left(\psi,\left[a;b\right]\right)=UTV\left(\psi,\left[a;b\right]\right)+DTV\left(\psi,\left[a;b\right]\right),
\]
\[
\psi\left(b\right)=\psi\left(a\right)+UTV\left(\psi,\left[a;b\right]\right)-DTV\left(\psi,\left[a;b\right]\right).
\]

Let $a\in\mathbb{R}.$ By $G[a;+\infty)$ we will denote the set of
real-valued regulated functions, defined on the interval $[a;+\infty),$
by $BV[a;+\infty)$ we will denote the subspace of $G[a;+\infty)$
consisting of functions of locally bounded variation and by $D[a;+\infty)$
we will denote the subspace of $G[a;+\infty)$ consisting of \emph{c\`{a}dl\`{a}g
functions}, i.e. right-continuous functions with left limits. For
$b>a$ by $\left\Vert \psi\right\Vert _{\left[a;b\right]},$ $\left\Vert \psi\right\Vert _{\left[a;b\right)}$
we will denote the semi-norms $\left\Vert \psi\right\Vert _{\left[a;b\right]}=\sup_{t\in\left[a;b\right]}\left|\psi\left(t\right)\right|,$
$\left\Vert \psi\right\Vert _{\left[a;b\right)}=\sup_{t\in\left[a;b\right)}\left|\psi\left(t\right)\right|.$

Now we define the play operator in the following way.

\begin{defn}\label{play_def} Let $c>0,$ $u\in G[0;+\infty)$ and
$\xi^{0}\in\left[u\left(0\right)-c/2;u\left(0\right)+c/2\right].$
There exists a unique, piecewise monotone function $\xi\in BV[0;+\infty)$
such that $\left\Vert \xi-u\right\Vert _{\left[0;+\infty\right)}\leq c/2,$
$\xi\left(0\right)=\xi^{0}$ and for any $t>0$ and $\chi:[0;+\infty)\rightarrow\mathbb{R}$
such that $\left\Vert \chi-u\right\Vert _{\left[0;t\right]}\leq c/2,$
$\chi\left(0\right)=\xi^{0},$ we have 
\[
TV\left(\xi,\left[0;t\right]\right)\leq TV\left(\chi,\left[0;t\right]\right).
\]
The map 
\[
G[0;+\infty)\ni u\mapsto\xi\in BV[0;+\infty)
\]
will be called the {\em play operator} and it will be denoted $\xi=\mathfrak{p}_{c/2}\left[\xi^{0},u\right].$
\end{defn}

The above definition stems from the laziness principle of the play
operator, firstly observed by V. Chernorutskii for continuous inputs,
namely that it associates with each function $u$ the function starting
from $\xi^{0}$ and of minimal total variation within the $c/2-$neighborhood
of $u$ in each subinterval $\left[0;t\right],$ $t>0$ of $[0;+\infty)$
(cf. \cite[Corollary 1.5]{K0}). The correctness of this definition
is guaranted by Corollary \ref{Play_UTVDTV}. The operator defined
in Corollary \ref{Play_UTVDTV} satisfies the conditions of Definition
\ref{play_def}. 

Now let us define 
\emph{truncated variation}, \emph{upward truncated variation} and
\emph{downward truncated variation} which, for a given $c>0,$ $0\leq a<b<+\infty$
and $\psi:[0;+\infty)\rightarrow\mathbb{R}$ are defined with the
formulas 
\[
TV^{c}\left(\psi,\left[a;b\right]\right):=\sup_{n}\sup_{a\leq t_{0}<t_{1}<\cdots<t_{n}\leq b}\sum_{i=1}^{n}\left(\left|\psi\left(t_{i}\right)-\psi\left(t_{i-1}\right)\right|-c\right)_{+},
\]
\[
UTV^{c}\left(\psi,\left[a;b\right]\right):=\sup_{n}\sup_{a\leq t_{0}<t_{1}<\cdots<t_{n}\leq b}\sum_{i=1}^{n}\left(\psi\left(t_{i}\right)-\psi\left(t_{i-1}\right)-c\right)_{+},
\]

\[
DTV^{c}\left(\psi,\left[a;b\right]\right):=\sup_{n}\sup_{a\leq t_{0}<t_{1}<\cdots<t_{n}\leq b}\sum_{i=1}^{n}\left(\psi\left(t_{i-1}\right)-\psi\left(t_{i}\right)-c\right)_{+}.
\]
\begin{rmk} It is easy to prove (cf. \cite[Fact 2.2]{L1}) that $\psi:\left[a;b\right]\rightarrow\mathbb{R}$
is regulated if and only if for every $c>0,$ $TV^{c}\left(\psi,\left[a;b\right]\right)<+\infty.$\end{rmk}
We have 
\begin{thm}\label{lowerBound-1} Let $\psi:[0;+\infty) \rightarrow \mathbb{R}$ and $0\leq a<b<+\infty.$
If $\chi:[0;+\infty)\rightarrow\mathbb{R}$ satisfies $\left\Vert \chi-\psi\right\Vert _{\left[a;b\right]}\leq c/2$
then 
\begin{equation} \label{a_lower_bound}
TV\left(\chi,\left[a;b\right]\right)\geq TV^{c}\left(\psi,\left[a;b\right]\right).
\end{equation}
Thus $TV^{c}\left(\psi,\left[a;b\right]\right)$ is a lower bound for
the total variation on the interval $\left[a;b\right]$ for all functions
$\chi:[0;+\infty)\rightarrow\mathbb{R}$ such that $\left\Vert \chi-\psi\right\Vert _{\left[a;b\right]}\leq c/2.$
Similarly 
\[
UTV\left(\chi,\left[a;b\right]\right)\geq UTV^{c}\left(\psi,\left[a;b\right]\right),\mbox{ }DTV\left(\chi,\left[a;b\right]\right)\geq DTV^{c}\left(\psi,\left[a;b\right]\right).
\]
Moreover, if $\psi\in G[a;+\infty)$ then there exists a piecewise monotone
function $\psi^{c}\in BV[a;+\infty)$ such that $\left\Vert \psi^{c}-\psi\right\Vert _{[a;+\infty)}\leq c/2$
and for any $t\geq a$ 
\begin{equation} \label{the_lower_bound}
TV\left(\psi^{c},\left[a;t\right]\right)=TV^{c}\left(\psi,\left[a;t\right]\right)<+\infty,
\end{equation}
\[
UTV\left(\psi^{c},\left[a;t\right]\right)=UTV^{c}\left(\psi,\left[a;t\right]\right),\mbox{ }DTV\left(\psi^{c},\left[a;t\right]\right)=DTV^{c}\left(\psi,\left[a;t\right]\right).
\]
\end{thm} 
From (\ref{a_lower_bound}) and (\ref{the_lower_bound}) we get that for $\psi\in G[a;+\infty)$ the truncated variation $TV^{c}\left(\psi,\left[a;b\right]\right)$ is {\em the greatest lower bound} for
the total variation on the interval $\left[a;b\right]$ for all functions
$\chi:[0;+\infty)\rightarrow\mathbb{R}$ such that $\left\Vert \chi-\psi\right\Vert _{\left[a;b\right]}\leq c/2.$
\begin{rmk}
From the last statement of Theorem \ref{lowerBound-1} and the Jordan decomposition we get that for any $\psi\in G[0;+\infty),$ $0\leq a <b$ and $c>0$
$$TV^{c}\left(\psi,\left[a;b\right]\right) = UTV^{c}\left(\psi,\left[a;b\right]\right) + DTV^{c}\left(\psi,\left[a;b\right]\right).$$
\end{rmk}
A special case of Theorem \ref{lowerBound-1}, when $\psi$ is a c\`{a}dl\`{a}g function, was proven in \cite{L1} (Theorem 3.6 and Theorem 4.1). Similar proof for regulated functions would require the consideration of many special cases; but here we omit the proof, since the first part of Theorem \ref{lowerBound-1} is a special 
case of Proposition \ref{lowerBound}, which will be proved in the
sequel and the second part follows from more general Theorem \ref{cond_tv_alpha_beta}
and Remark \ref{gamma}. 

From Theorem \ref{lowerBound-1}
and Definition \ref{play_def} for any $u\in G[0;+\infty)$ and $c,t>0$
we naturally have 
\begin{equation}
TV^{c}\left(u,\left[0;t\right]\right)=\inf_{\xi^{0}\in\left[u\left(0\right)-c/2;u\left(0\right)+c/2\right]}TV\left(\mathfrak{p}_{c/2}\left[\xi^{0},u\right],\left[0;t\right]\right)\label{eq:inf_tv1}
\end{equation}
and there exists appropriate starting value $\xi^{c}\in\left[u\left(0\right)-c/2;u\left(0\right)+c/2\right]$
such that for $u^{c}=\mathfrak{p}_{c/2}\left[\xi^{c},u\right]$ we
have 
\begin{equation}
TV^{c}\left(u,\left[0;t\right]\right)=TV\left(\mathfrak{p}_{c/2}\left[\xi^{c},u\right],\left[0;t\right]\right)=TV\left(u^{c},\left[0;t\right]\right),\label{eq:inf_tv2}
\end{equation}
and similar relations hold for $UTV$ and $DTV.$ On the other hand,
it is not difficult to prove (see Remark \ref{bounds}) that for any
given $\xi^{0}\in\left[u\left(0\right)-c/2;u\left(0\right)+c/2\right]$
\begin{equation}
TV\left(\mathfrak{p}_{c/2}\left[\xi^{0},u\right],\left[0;t\right]\right)\leq TV^{c}\left(u,\left[0;t\right]\right)+c\label{eq:sup_tv1}
\end{equation}
and similar relations hold for $UTV$ and $DTV.$ 

Formulas (\ref{eq:formula1}) and (\ref{eq:formula2}) show that it
is possible to invert the problem and using (upward, downward) truncated
variation represent the play operator for any regulated input $u\in G[0;+\infty).$
In fact $UTV^{-c^{0},c^{0}}$ and $DTV^{-c^{0},c^{0}}$ are upward
and downward truncated variations for more general than constant characteristics.
\subsection{The Skorohod problem}
When the input function of the play operator is regulated, then the
output, though with locally bounded variation, may still be only a
regulated function, no right- neither left-continuous. However, when
the input is a c\`{a}dl\`{a}g function then the output is also a c\`{a}dl\`{a}g function
$\xi$ with locally bounded variation. In such a case we may identify
the output with a finite, signed measure $\nu$ on Borel subsets of
$[0;+\infty)$ via the equality 
\[
\nu\left(a;b\right]=\xi\left(b\right)-\xi\left(a\right)
\]
for any $0\leq a\leq b,$ and denote $\nu=\mathrm{d}\xi.$ In this
case the output of the play operator is related to the \emph{Skorohod
map} on $\left[-c/2;c/2\right].$ The definition of the Skorohod map
on $\left[-c/2;c/2\right]$ is given below. 

Let $I[0;+\infty)$ denote
the set of real-valued, c\`{a}dl\`{a}g, non-decreasing functions, defined on
the interval $[0;+\infty).$

A pair of functions $\left(\phi,-\xi\right)\in D[0;+\infty)\times BV[0;+\infty)$
is said to be a solution of the {\em Skorohod problem} on $\left[-c/2,c/2\right]$
with starting condition  $\xi(0)=\xi^{0}$ for $u\in D[0;+\infty)$
if the following conditions are satisfied: 
\begin{enumerate}
\item for every $t\geq0,$ $\phi\left(t\right)=u\left(t\right)-\xi\left(t\right)\in\left[-c/2,c/2\right];$ 
\item $\xi=\xi_{u}-\xi_{d},$ where $\xi_{u},\xi_{d}\in I[0;+\infty)$ and
the corresponding measures $\mathrm{d}\xi_{u},$ $\mathrm{d}\xi_{d}$
are carried by $\left\{ t\geq0:\phi\left(t\right)=c/2\right\} $ and
$\left\{ t\geq0:\phi\left(t\right)=-c/2\right\} $ respectively; 
\item \textbf{$\xi(0)=\xi^{0}.$} 
\end{enumerate}
For $\xi^{0}\in\left[u\left(0\right)-c/2;u\left(0\right)+c/2\right]$
the Skorohod problem has a unique solution and we have $\xi=\mathfrak{p}_{c/2}\left[\xi^{0},u\right]$.
With the notation of Corollary \ref{Play_UTVDTV} we also have 
\[
\mathrm{d}\xi_{u}=\mathrm{d}UTV^{-c^{0},c^{0}}\left(u^{\xi^{0}},\left[0;\cdot\right]\right)\mbox{ and }\mathrm{d}\xi_{d}=\mathrm{d}DTV^{-c^{0},c^{0}}\left(u^{\xi^{0}},\left[0;\cdot\right]\right).
\]
Thus, for any measurable function $\Phi$ with $\left\Vert \Phi\right\Vert _{\left[0;+\infty\right)}\leq1$
\begin{equation}
\int_{0}^{\cdot}\Phi\mathrm{d}\xi\leq TV\left(\xi,\left[0;\cdot\right]\right)=TV^{-c^{0},c^{0}}\left(u^{\xi^{0}},\left[0;\cdot\right]\right)\label{eq:variation}
\end{equation}
($\int$ denotes here the Lebesgue-Stieltjes integral). From condition
(2) of the Skorohod problem we get that $\mathrm{d}\xi_{u}$ and $\mathrm{d}\xi_{l}$
are mutually singular measures and equality in (\ref{eq:variation})
holds for $\Phi=2\phi/c,$ indeed 
\begin{align*}
\frac{c}{2}\cdot TV\left(\xi,\left[0;\cdot\right]\right) & =\int_{0}^{\cdot}\left(u-\xi\right)\mathrm{d}\xi_{u}-\int_{0}^{\cdot}\left(u-\xi\right)\mathrm{d}\xi_{d}\\
 & =\int_{0}^{\cdot}\left(u\left(s\right)-\xi\left(s\right)\right)\mathrm{d}\xi\left(s\right)=\int_{0}^{\cdot}\phi\left(s\right)\mathrm{d}\xi\left(s\right).
\end{align*}
On the other hand, 
\begin{align*}
\frac{c}{2}\xi & =\int_{0}^{\cdot}\left(u-\xi\right)\mathrm{d}\xi_{u}+\int_{0}^{\cdot}\left(u-\xi\right)\mathrm{d}\xi_{d}\\
 & =\int_{0}^{\cdot}\left(u\left(s\right)-\xi\left(s\right)\right)\mathrm{d}TV\left(\xi,\left[0;s\right]\right)=\int_{0}^{\cdot}\phi\left(s\right)\mathrm{d}TV\left(\xi,\left[0;s\right]\right).
\end{align*}

Let us mention that for a c\`{a}dl\`{a}g input and c\`{a}dl\`{a}g  
functions $\alpha,$ $\beta$ it is possible to define 
the \emph{Skorohod map on }$\left[\alpha;\beta\right].$
The definition of\emph{ }the Skorohod map\emph{ }on $\left[\alpha;\beta\right]$ (see e.g. \cite[Definition 2.1]{BurdzySPA})
is similar to the definition of the Skorohod map on $\left[-c/2;c/2\right],$
with $-c/2$ replaced by $\alpha$ and $c/2$ replaced by $\beta.$ In \cite[Definition 2.1]{BurdzySPA}
the starting condition $\xi^{0}$ is not specified, but it may be
imposed with the shift of the input - in \cite[Definition 2.1]{BurdzySPA} one substitutes  $\psi$ with 
$\psi-\psi\left(0\right)+\xi^{0}.$

\section{The play operator with time-dependent boundaries}

Now we will extend the definition of the play operator to the case
when the interval $\left[-c/2;c/2\right]$ is replaced by time-dependent,
characteristics (boundaries) $\alpha,\beta\in G\left[0;+\infty\right)$
such that for all $t\geq0,$ $\alpha\left(t\right)\leq\beta\left(t\right)$ ($\alpha \leq \beta$ in short).
In this section we will focus on the case when the output has locally
finite variation. We start with necessary definitions and notation.

Assume that $\alpha,$ $\beta\in G\left[0;+\infty\right),$ $\alpha \leq \beta,$ are given and fixed.
For $\psi\in G\left[0;+\infty\right)$ let us denote its symmetrization
with respect to $\alpha$ and $\beta$ by $\tilde{\psi},$ i.e. $\tilde{\psi}=\psi-\tfrac{1}{2}\left(\alpha+\beta\right).$
Moreover, let us define $\gamma:=\beta-\alpha.$

Now, for $0\leq a<b<+\infty$ we consider $\alpha,\beta-$truncated
variation of $\psi$ on $\left[a;b\right]$ defined as 

\begin{eqnarray*}
TV^{\alpha,\beta}\left(\psi,\left[a;b\right]\right)
 :=\sup_{n}\sup_{a\leq t_{0}<t_{1}<\cdots<t_{n}\leq b}\sum_{i=1}^{n}\left(\left|\tilde{\psi}\left(t_{i}\right)-\tilde{\psi}\left(t_{i-1}\right)\right|-\frac{1}{2}\left(\gamma\left(t_{i}\right)+\gamma\left(t_{i-1}\right)\right)\right)_{+}.
\end{eqnarray*}

Moreover, we define $\alpha,\beta-$ upward and downward truncated
variations: 
\begin{eqnarray*}
  UTV^{\alpha,\beta}\left(\psi,\left[a;b\right]\right)
  :=\sup_{n}\sup_{a\leq t_{0}<t_{1}<\cdots<t_{n}\leq b}\sum_{i=1}^{n}\left(\tilde{\psi}\left(t_{i}\right)-\tilde{\psi}\left(t_{i-1}\right)-\frac{1}{2}\left(\gamma\left(t_{i}\right)+\gamma\left(t_{i-1}\right)\right)\right)_{+},
\end{eqnarray*}
\begin{eqnarray*}
  DTV^{\alpha,\beta}\left(\psi,\left[a;b\right]\right)
  :=\sup_{n}\sup_{a\leq t_{0}<t_{1}<\cdots<t_{n}\leq b}\sum_{i=1}^{n}\left(\tilde{\psi}\left(t_{i-1}\right)-\tilde{\psi}\left(t_{i}\right)-\frac{1}{2}\left(\gamma\left(t_{i}\right)+\gamma\left(t_{i-1}\right)\right)\right)_{+}.
\end{eqnarray*}

For $c>0,$ $\alpha\equiv-c/2,$ $\beta\equiv c/2$ we naturally have
$TV^{\alpha,\beta}\left(\psi,\left[a;b\right]\right)=TV^{c}\left(\psi,\left[a;b\right]\right)$
and the same holds for $UTV^{\alpha,\beta}$ and $DTV^{\alpha,\beta}.$ 

We have the following simple but important result:
\begin{prop}\label{lowerBound} Let $\psi:[0;+\infty)\rightarrow\mathbb{R},$
$\xi:[0;+\infty)\rightarrow\mathbb{R}$ and assume that for all $t\geq0,$
the inequalities $\alpha\left(t\right)\leq \psi\left(t\right)-\xi\left(t\right)\leq\beta\left(t\right)$
($\alpha\leq \psi-\xi\leq\beta$ in short) hold. Then for $0\leq a<b<+\infty$
\[
TV\left(\xi,\left[a;b\right]\right)\geq TV^{\alpha,\beta}\left(\psi,\left[a;b\right]\right),
\]
thus $TV^{\alpha,\beta}\left(\psi,\left[a;b\right]\right)$ are the
lower bounds for total variation on the intervals $\left[a;b\right],$
$0\leq a<b<+\infty,$ for all functions $\xi:[0;+\infty)\rightarrow\mathbb{R}$
such that $\alpha\leq\psi-\xi\leq\beta.$ Similarly 
\[
UTV\left(\xi,\left[a;b\right]\right)\geq UTV^{\alpha,\beta}\left(\psi,\left[a;b\right]\right),\mbox{ }DTV\left(\xi,\left[a;b\right]\right)\geq DTV^{\alpha,\beta}\left(\psi,\left[a;b\right]\right).
\]
\end{prop} \begin{proof} The first assertion follows easily from the
following inequalities. If $\alpha\leq\psi-\xi\leq\beta$ and $0\leq s<t<+\infty$
then 
\[
\xi\left(t\right)-\xi\left(s\right)\geq\left(\psi-\beta\right)\left(t\right)-\left(\psi-\alpha\right)\left(s\right)=\tilde{\psi}\left(t\right)-\tilde{\psi}\left(s\right)-\frac{1}{2}\left(\gamma\left(s\right)+\gamma\left(t\right)\right).
\]
Similarly, 
\[
\xi\left(s\right)-\xi\left(t\right)\geq\left(\psi-\beta\right)\left(s\right)-\left(\psi-\alpha\right)\left(t\right)=\tilde{\psi}\left(s\right)-\tilde{\psi}\left(t\right)-\frac{1}{2}\left(\gamma\left(s\right)+\gamma\left(t\right)\right).
\]
Hence 
\[
\left|\xi\left(t\right)-\xi\left(s\right)\right|\geq\left(\left|\tilde{\psi}\left(s\right)-\tilde{\psi}\left(t\right)\right|-\frac{1}{2}\left(\gamma\left(s\right)+\gamma\left(t\right)\right)\right)_{+}.
\]
The second assertion follows directly from the inequalities 
\[
\left(\xi\left(t\right)-\xi\left(s\right)\right)_{+}\geq\left(\tilde{\psi}\left(t\right)-\tilde{\psi}\left(s\right)-\frac{1}{2}\left(\gamma\left(s\right)+\gamma\left(t\right)\right)\right)_{+},
\]
\[
\left(\xi\left(s\right)-\xi\left(t\right)\right)_{+}\geq\left(\tilde{\psi}\left(s\right)-\tilde{\psi}\left(t\right)-\frac{1}{2}\left(\gamma\left(s\right)+\gamma\left(t\right)\right)\right)_{+}.
\]
\end{proof} 
From Proposition \ref{lowerBound} we get immediately
the necessary condition for the output $\xi$ of any (not only play)
operator with input $u$ and such that $\alpha\leq u-\xi\leq\beta,$
to have locally bounded variation: 
\begin{equation}
\mbox{for any }t\geq0,\mbox{ }TV^{\alpha,\beta}\left(u,\left[0;t\right]\right)<+\infty.\label{eq:cond_finiteness}
\end{equation}

Now, from the laziness principle we derive the definition of the play
operator with regulated inputs and time-dependent, regulated boundaries.\begin{defn}
\label{play_def-1} Let $\alpha,\beta,u\in G\left[0;+\infty\right),$ $\alpha \leq \beta,$
$\xi^{0}\in$$\left[u\left(0\right)-\beta\left(0\right);u\left(0\right)-\alpha\left(0\right)\right]$
and assume that the condition (\ref{eq:cond_finiteness}) holds. There
exists an unique regulated function $\xi\in BV[0;+\infty)$ such that
$\xi\left(0\right)=\xi^{0}$ and for any $t>0$ and $\chi:[0;+\infty)\rightarrow\mathbb{R}$
such that $\chi\left(0\right)=\xi^{0}$ and $\alpha\leq u-\chi\leq\beta$
we have 
\[
TV\left(\xi,\left[0;t\right]\right)\leq TV\left(\chi,\left[0;t\right]\right).
\]
The map 
\[
G[0;+\infty)\ni u\mapsto\xi\in BV[0;+\infty)
\]
will be called the \emph{play operator with boundaries (or characteristics)
$\alpha,$ $\beta$} and it will be denoted $\xi=\mathfrak{p}_{\alpha,\beta}\left[\xi^{0},u\right].$
\end{defn}

 The correctness of this definition is guaranteed by Corollary
\ref{Play_UTVDTV-1} which is an easy consequence of the following main result of this paper. 
\begin{thm} \label{cond_tv_alpha_beta} Let $t_{0} \in [0; +\infty),$
$\psi,\alpha,\beta\in G[t_{0};+\infty)$ and assume that for any $t\geq t_{0},$
$\alpha\left(t\right)\leq\beta\left(t\right)$ and 
\begin{equation}
TV^{\alpha,\beta}\left(\psi,\left[t_{0};t\right]\right)<+\infty.\label{eq:finiteness}
\end{equation}
There exists a function $\psi^{\alpha,\beta}\in BV[t_{0};+\infty)$
such that for all $t\geq t_{0},$ $\alpha\left(t\right)\leq\psi\left(t\right)-\psi^{\alpha,\beta}\left(t\right)\leq\beta\left(t\right)$
and 
\begin{equation} \label{eq:Jordan1-0}
TV\left(\psi^{\alpha,\beta},\left[t_{0};t\right]\right)=TV^{\alpha,\beta}\left(\psi,\left[a;t\right]\right)<+\infty,
\end{equation}
\begin{equation}
UTV\left(\psi^{\alpha,\beta},\left[t_{0};t\right]\right)=UTV^{\alpha,\beta}\left(\psi,\left[t_{0};t\right]\right),\mbox{ }DTV\left(\psi^{\alpha,\beta},\left[t_{0};t\right]\right)=DTV^{\alpha,\beta}\left(\psi,\left[t_{0};t\right]\right),\label{eq:Jordan1-1}
\end{equation}
and for $t\geq t_{0},$ $\psi^{\alpha,\beta}\left(t\right)$ may be
expressed as 
\begin{equation}
\psi^{\alpha,\beta}\left(t\right)=\psi^{\alpha,\beta}\left(t_{0}\right)+UTV^{\alpha,\beta}\left(\psi,\left[t_{0};t\right]\right)-DTV^{\alpha,\beta}\left(\psi,\left[t_{0};t\right]\right).\label{eq:Jordan-1}
\end{equation}
Moreover, for any $\xi\in G[t_{0};+\infty)$ such that 
\begin{enumerate}
\item $\xi\left(t_{0}\right)=\psi^{\alpha,\beta}\left(t_{0}\right),$ 
\item for any $t\ge t_{0},$ $\alpha\left(t\right)\leq\psi\left(t\right)-\xi\left(t\right)\leq\beta\left(t\right)$
and $TV\left(\xi,\left[t_{0};t\right]\right)\leq TV^{\alpha,\beta}\left(\psi,\left[t_{0};t\right]\right),$ 
\end{enumerate}
we have $\xi\left(t\right)=\psi^{\alpha,\beta}\left(t\right)$ for
$t\geq t_{0}.$ \end{thm} 

\begin{proof} The proof consists of two steps. 

\textbf{Step 1. Proof for step functions.} First we will assume that
the input $\psi$ and the boundaries $\alpha,$ $\beta$ are step
functions, i.e. $\psi$ has representation 
\begin{equation}
\psi\left(t\right)=\sum_{k=0}^{\infty}\psi_{2k}1_{\left\{ t_{k}\right\} }\left(t\right)+\sum_{k=0}^{\infty}\psi_{2k+1}1_{\left(t_{k};t_{k+1}\right)}\left(t\right),\label{eq:psi_representation}
\end{equation}
where $t_{0}<t_{1}<\ldots$ and $\lim_{k\rightarrow+\infty}t_{k}=+\infty,$
and similar representations (with the same $t_{i},$ $i=0,1,\ldots$)
hold for $\alpha$ and $\beta.$ We will simply construct the appropriate
function $\psi^{\alpha,\beta}.$ First, for $i=0,1,\ldots$ we define
\[
I_{U}\left(i\right)=\min\left\{ j\geq i:\min_{i\leq k\leq j}\left(\psi_{k}-\alpha_{k}\right)<\left(\psi_{j}-\beta_{j}\right)\right\} ,
\]
\[
I_{D}\left(i\right)=\min\left\{ j\geq i:\max_{i\leq k\leq j}\left(\psi_{k}-\beta_{k}\right)>\left(\psi_{j}-\alpha_{j}\right)\right\} .
\]
Without loss of generality we may assume that $I_{U}\left(0\right)\leq I_{D}\left(0\right),$ since the case $I_{D}\left(0\right)\leq I_{U}\left(0\right)$ is symmetric. 
Since for every $j\geq0,$ $\psi_{j}-\alpha_{j}\geq\psi_{j}-\beta_{j},$
the value $\psi_{I_{U}\left(i\right)}-\alpha_{I_{U}\left(i\right)}$
can not be a new minimum to date of the sequence $\left(\psi_{j}-\alpha_{j}\right)_{j\geq i}$
but rather the value $\psi_{I_{U}\left(i\right)}-\beta_{I_{U}\left(i\right)}$
is a new maximum to date of the sequence $\left(\psi_{j}-\beta_{j}\right)_{j\geq i}.$

Assuming that $I_{U}\left(0\right)\leq I_{D}\left(0\right)$ we define
sequences $\left(I_{D,k}\right)_{k=-1}^{\infty},\mbox{ }\left(I_{U,k}\right)_{k=0}^{\infty}$
in the following way: $I_{D,-1}=0,$ $I_{U,0}=I_{U}\left(0\right)$
and for $k=0,1,2,\cdots$ 
\[
I_{D,k}=\left\{ \begin{array}{lr}
I_{D}\left(I_{U,k}\right) & \text{ if }I_{U,k}<+\infty,\\
+\infty & \mbox{otherwise},
\end{array}\right.I_{U,k+1}=\left\{ \begin{array}{lr}
I_{U}\left(I_{D,k}\right) & \text{ if }I_{D,k}<+\infty,\\
+\infty & \mbox{otherwise}.
\end{array}\right.
\]
We also define $m_{-1}=\min_{0\leq j<I_{U,0}}\left(\psi_{j}-\alpha_{j}\right).$
Further, for $k=0,1,\cdots,$ and $i\geq I_{D,k},$ let us denote $m_{k}\left(i\right)=\min_{I_{D,k}\leq j\leq i}\left(\psi_{j}-\alpha_{j}\right)$
and for $k=0,1,\cdots,$ $i\geq I_{U,k},$ $M_{k}\left(i\right)=\max_{I_{U,k}\leq j\leq i}\left(\psi_{j}-\beta_{j}\right).$ 

Now, assuming that $I_{U}\left(0\right)\leq I_{D}\left(0\right),$
we define auxiliary function $p^{\alpha,\beta}$ defined on the set
of non-negative integers $i=0,1,2,\ldots$ as follows. 
\[
p_{i}^{\alpha,\beta}=\left\{ \begin{array}{lr}
m_{-1} & \text{ if }0\leq i<I_{U,0};\\
M_{k}\left(i\right) & \text{ if }I_{U,k}\leq i<I_{D,k},\mbox{ }k=0,1,\ldots;\\
m_{k}\left(i\right) & \text{ if }I_{D,k}\leq i<I_{U,k+1},\mbox{ }k=0,1,\ldots.
\end{array}\right.
\]
\begin{rmk} The value of $m_{-1}$ is chosen in such a way that the
difference $\psi-m_{-1}$ stays within the values  $\alpha(t)$ and $\beta(t),$
$t\geq t_{0},$ for the longest time possible. \end{rmk}

Now we will prove that for $i=0,1,\ldots,$ $\alpha_{i}\leq\psi_{i}-p_{i}^{\alpha,\beta}\leq\beta_{i}.$
It is enough to consider few cases. 
\begin{itemize}
\item $0\leq i<I_{U,0}=I_{U}\left(0\right).$ In this case, since $0\leq i<I_{U,0},$
$m_{-1}\leq\psi_{i}-\alpha_{i}$ and 
\[
\psi_{i}-p_{i}^{\alpha,\beta}=\psi_{i}-m_{-1}\geq\psi_{i}-\left(\psi_{i}-\alpha_{i}\right)=\alpha_{i}.
\]
On the other hand, since $i<I_{U,0}\leq I_{D}\left(0\right),$ $m_{-1}\geq\psi_{i}-\beta_{i}$
and 
\[
\psi_{i}-p_{i}^{\alpha,\beta}=\psi_{i}-m_{-1}\leq\psi_{i}-\left(\psi_{i}-\beta_{i}\right)=\beta_{i}.
\]

\item $I_{U,k}\leq i<I_{D,k},\mbox{ }k=0,1,\ldots,$ for some $k=0,1,2,\cdots$
In this case, by the definition of $I_{D,k}$, $\max_{I_{U,k}\leq j\leq i}\left(\psi_{j}-\beta_{j}\right)$
belongs to the interval $\left[\psi_{i}-\beta_{i};\psi_{i}-\alpha_{i}\right],$
hence 
\[
\psi_{i}-p_{i}^{\alpha,\beta}=\psi_{i}-\max_{I_{U,k}\leq j\leq i}\left(\psi_{j}-\beta_{j}\right)\in\left[\alpha_{i};\beta_{i}\right].
\]

\item $I_{D,k}\leq i<I_{U,k+1},\mbox{ }k=0,1,\ldots$ for some $k=0,1,2,\cdots$
In this case, by the definition of $T_{U,k+1}$, $\min_{I_{D,k}\leq j\leq i}\left(\psi_{j}-\alpha_{j}\right)$
belongs to the interval $\left[\psi_{i}-\beta_{i};\psi_{i}-\alpha_{i}\right],$
hence 
\[
\psi_{i}-p_{i}^{\alpha,\beta}=\psi_{i}-\min_{I_{D,k}\leq j\leq i}\left(\psi_{j}-\alpha_{j}\right)\in\left[\alpha_{i};\beta_{i}\right].
\]

\end{itemize}
Now we define the function $\psi^{\alpha,\beta}$ in the following
way. Let $s\geq t_{0}$ and $i\left(s\right)=0,1,2,\ldots$ be the
unique integer such that $s=t_{i\left(s\right)}$ or $s\in\left(t_{i\left(s\right)};t_{i\left(s\right)+1}\right),$
we set 
\begin{equation}
\psi^{\alpha,\beta}\left(s\right)=\begin{cases}
p_{2i\left(s\right)}^{\alpha,\beta} & \mbox{if }s=t_{i\left(s\right)};\\
p_{2i\left(s\right)+1}^{\alpha,\beta} & \mbox{if }s\in\left(t_{i\left(s\right)};t_{i\left(s\right)+1}\right).
\end{cases}\label{eq:discr_representation}
\end{equation}
By the representation (\ref{eq:psi_representation}) we get $\alpha\left(t\right)\leq\psi\left(t\right)-\psi^{\alpha,\beta}\left(t\right)\leq\beta\left(t\right)$
for $t\geq t_0.$ By this and by Proposition \ref{lowerBound} (the minimality
of $TV^{\alpha,\beta}\left(\psi,\left[t_{0};t\right]\right)$) we
have 
\begin{equation}
TV^{\alpha,\beta}\left(\psi,\left[t_{0};t\right]\right)\leq TV\left(\psi^{\alpha,\beta},\left[t_{0};t\right]\right)<+\infty.\label{eq:lower_bound}
\end{equation}
To prove that $\psi^{\alpha,\beta}$ has the smallest variation possible on the intervals of the form $\left[t_{0};t\right],$ $t\geq t_{0},$
equal $TV^{\alpha,\beta}\left(\psi,\left[t_{0};t\right]\right),$
we will again use the discrete representation (\ref{eq:discr_representation}).
For sequences $p=\left(\psi_{i}\right)_{i\geq0},$ $a=\left(\alpha_{i}\right)_{i\geq0}$
and $b=\left(\beta_{i}\right)_{i\geq0}$ such that $\alpha_{i}\leq\beta_{i},$
$i=0,1,2,\ldots$ we denote $\gamma_{i}=\beta_{i}-\alpha_{i},$ $\tilde{\psi}_{i}=\psi_{i}-\frac{1}{2}\left(\alpha_{i}+\beta_{i}\right)$
and for $k\leq l,$ $k,l=0,1,2,\ldots$ we define the discrete version
of $\alpha,\beta-$truncated variation 
\begin{align*}
 & TV^{a,b}\left(p,\left[k;l\right]\right):=\sup\sum_{j=1}^{n}\left(\left|\tilde{\psi}_{m_{j}}-\tilde{\psi}_{m_{j-1}}\right|-\frac{1}{2}\left(\gamma_{m_{j-1}}+\gamma_{m_{j}}\right)\right)_{+}\\
 & =\sup\sum_{j=1}^{n}\left\{ \left(\max\left(\psi_{m_{j}}-\beta_{m_{j}}\right)-\left(\psi_{m_{j-1}}-\alpha_{m_{j-1}}\right)\right)_{+},\left(\left(\psi_{m_{j-1}}-\beta_{m_{j-1}}\right)-\left(\psi_{m_{j}}-\alpha_{m_{j}}\right)\right)_{+}\right\} ,
\end{align*}
where $"\sup"$ stands for $"\sup_{n}\sup_{k\leq m_{1}<m_{2}<\ldots<m_{n}\leq l}".$
Similarly, we define $UTV^{a,b}\left(p,\left[k;l\right]\right)$ and
$DTV^{a,b}\left(p,\left[k;l\right]\right),$ and set $TV=TV^{0,0},$
$UTV=UTV^{0,0},$ and $DTV=DTV^{0,0}.$ By the representations (\ref{eq:psi_representation}),
(\ref{eq:discr_representation}) we have that 
\begin{equation}
TV^{\alpha,\beta}\left(\psi,\left[t_{0};t\right]\right)=TV^{a,b}\left(p,\left[0;i\left(t\right)\right]\right),\label{eq:discr1}
\end{equation}
\begin{equation}
TV\left(\psi^{\alpha,\beta},\left[t_{0};t\right]\right)=TV\left(p^{\alpha,\beta},\left[0;i\left(t\right)\right]\right)\label{eq:discr2}
\end{equation}
and similar equalities hold for $UTV$ and $DTV.$

We have the discrete counterpart of Proposition \ref{lowerBound},
stating that for any $h=\left(\chi_{i}\right)_{i\geq0}$ such that
$\alpha_{i}\leq\psi_{i}-\chi_{i}\leq\beta_{i},$ $i=1,2,\ldots,$
for $k\leq l,$ $k,l=0,1,2,\ldots$ we have 
\[
TV\left(h,\left[k;l\right]\right)\geq TV^{a,b}\left(p,\left[k;l\right]\right)
\]
and similar bounds hold for $UTV$ and $DTV.$ We will prove that
$p^{\alpha,\beta}$ has the smallest possible variation equal $TV^{a,b}\left(p,\left[0;i\right]\right),$
on the intervals of the form $\left[0;i\right],$ $i=0,1,\ldots$
among such functions. Again, we consider several cases
\begin{itemize}
\item $0\leq i<I_{U,0}.$ In this case we naturally have $TV\left(p^{\alpha,\beta},\left[0;i\right]\right)=0\leq TV^{a,b}\left(p,\left[0;i\right]\right).$
\item $I_{U,0}\leq i<I_{D,0}.$ In this case we have 
\begin{align*}
TV\left(p^{\alpha,\beta},\left[0;i\right]\right) & =p_{i}^{\alpha,\beta}-m_{-1}\\
 & =\max_{I_{U,0}\leq j\leq i}\left(\psi_{j}-\beta_{j}\right)-\min_{0\leq j<I_{U,0}}\left(\psi_{j}-\alpha_{j}\right)\leq TV^{a,b}\left(p,\left[0;i\right]\right).
\end{align*}

\item $i\geq I_{D,0}$. We consider positive and negative variations of $p^{\alpha,\beta}$
on the intervals of the form $\left[I_{U,k}-1;I_{D,k}-1\right],$
$\left[I_{D,k}-1;I_{U,k+1}-1\right],$ $k=0,1,\ldots.$ Since $p^{\alpha,\beta}$
is non-decreasing on $\left[I_{U,k}-1;I_{D,k}-1\right]$ and $I_{D,k}\geq I_{U,k}+1,$
$k=0,1,\ldots,$ it is easy to see that for any $i\in\left[I_{U,k};I_{D,k}-1\right]$
we have 
\begin{align*}
TV\left(p^{\alpha,\beta},\left[I_{U,k}-1;i\right]\right) & =p_{i}^{\alpha,\beta}-p_{I_{U,k}-1}^{\alpha,\beta}\\
=\max_{I_{U,k}\leq j\leq i}\left(\psi_{j}-\beta_{j}\right)- & \min_{I_{D,k-1}\leq j\leq I_{U,k}-1}\left(\psi_{j}-\alpha_{j}\right)\leq TV^{a,b}\left(p,\left[I_{U,k}-1;i\right]\right).
\end{align*}
Similarly, since $p^{\alpha,\beta}$ is non-increasing on $\left[I_{D,k}-1;I_{U,k+1}-1\right],$
$k=0,1,\ldots,$ for any $i\in\left[I_{D,k};I_{U,k+1}-1\right]$ we
have 
\begin{align*}
TV\left(p^{\alpha,\beta},\left[I_{D,k}-1;i\right]\right) & =p_{I_{D,k}-1}^{\alpha,\beta}-p_{i}^{\alpha,\beta}\leq TV^{a,b}\left(p,\left[I_{D,k}-1;i\right]\right),
\end{align*}
Summing up these variations, for $i\in\left[I_{U,k};I_{D,k}-1\right]$
we get 
\begin{align}
TV\left(p^{\alpha,\beta},\left[0;i\right]\right) & \leq\sum_{j=0}^{k-1}TV^{a,b}\left(p,\left[I_{U,j}-1;I_{D,j}-1\right]\right)+\sum_{j=0}^{k-1}TV^{a,b}\left(p,\left[I_{D,j}-1;I_{U,j+1}-1\right]\right)+TV^{a,b}\left(p,\left[I_{U,k}-1;i\right]\right)\nonumber \\
 & \leq TV^{a,b}\left(p,\left[0;i\right]\right).\label{eq:discr_estimate}
\end{align}
The last inequality follows from the subadditivity of $TV^{a,b},$
i.e. $TV^{a,b}\left(p,\left[k;l\right]\right)+TV^{a,b}\left(p,\left[l;m\right]\right)\leq TV^{a,b}\left(p,\left[k;m\right]\right)$
for $0\leq k\leq l\leq m<+\infty,$ $k,l,m=0,1,2,\ldots$. Similarly,
we get the same estimate (\ref{eq:discr_estimate}) for $i\in\left[I_{D,k};I_{U,k+1}-1\right].$
\end{itemize}
Hence, for every $i=0,1,\ldots$ we have proved the estimate $TV\left(p^{\alpha,\beta},\left[0;i\right]\right)\leq TV^{a,b}\left(p,\left[0;i\right]\right).$
From this and (\ref{eq:discr1}), (\ref{eq:discr2}), for any $t\geq t_{0}$
we finally get 
\[
TV\left(\psi^{\alpha,\beta},\left[t_{0};t\right]\right)\leq TV^{\alpha,\beta}\left(\psi,\left[t_{0};t\right]\right).
\]
Now, by Proposition \ref{eq:lower_bound} and by the estimate 
\[
TV^{\alpha,\beta}\left(\psi,\left[t_{0};t\right]\right)\leq UTV^{\alpha,\beta}\left(\psi,\left[t_{0};t\right]\right)+DTV^{\alpha,\beta}\left(\psi,\left[t_{0};t\right]\right),
\]
which is an easy consequence of the definition of $TV^{\alpha,\beta},$
$UTV^{\alpha,\beta}$ and $DTV^{\alpha,\beta},$ and we arrive at
\begin{align*}
 & UTV^{\alpha,\beta}\left(\psi,\left[t_{0};t\right]\right)+DTV^{\alpha,\beta}\left(\psi,\left[t_{0};t\right]\right)
 \leq UTV\left(\psi^{\alpha,\beta},\left[t_{0};t\right]\right)+DTV\left(\psi^{\alpha,\beta},\left[t_{0};t\right]\right)\\
 & =TV\left(\psi^{\alpha,\beta},\left[t_{0};t\right]\right)\leq TV^{\alpha,\beta}\left(\psi,\left[t_{0};t\right]\right)
  \leq UTV^{a,b}\left(p,\left[0;i\right]\right)+DTV^{a,b}\left(p,\left[0;i\right]\right).
\end{align*}
Hence, for any $t\geq t_{0}$ we must have 
\[
UTV\left(\psi^{\alpha,\beta},\left[t_{0};t\right]\right)=UTV^{\alpha,\beta}\left(\psi,\left[t_{0};t\right]\right),\mbox{ }DTV\left(\psi^{\alpha,\beta},\left[t_{0};t\right]\right)=DTV^{\alpha,\beta}\left(\psi,\left[t_{0};t\right]\right)
\]
and we get (\ref{eq:Jordan1-1}) 
\[
TV\left(\psi^{\alpha,\beta},\left[t_{0};t\right]\right) = TV^{\alpha,\beta}\left(\psi,\left[t_{0};t\right]\right)=UTV^{\alpha,\beta}\left(\psi,\left[t_{0};t\right]\right)+DTV^{\alpha,\beta}\left(\psi,\left[t_{0};t\right]\right).
\]
Now, (\ref{eq:Jordan-1}) follows the Jordan decomposition of $\psi^{\alpha,\beta}.$ 

Finally, to prove the last assertion about the function $\xi$, we
reason similarly as for $\psi^{\alpha,\beta},$ and by the minimality
of $UTV^{\alpha,\beta}\left(\psi,\left[t_{0};t\right]\right),$ $DTV^{\alpha,\beta}\left(\psi,\left[t_{0};t\right]\right)$
we obtain the Jordan decomposition for $t\geq a,$ 
\[
\xi\left(t\right)=\xi\left(t_{0}\right)+UTV^{\alpha,\beta}\left(\psi,\left[t_{0};t\right]\right)-DTV^{\alpha,\beta}\left(\psi,\left[t_{0};t\right]\right).
\]
Since the starting values $\xi\left(t_{0}\right)$ and $\psi^{\alpha,\beta}\left(t_{0}\right)$
are equal, we get $\xi\left(t\right)=\psi^{\alpha,\beta}\left(t\right)$
for all $t\geq t_{0}.$

\textbf{Step 2. Proof for arbitrary regulated functions.} Let us fix
$c>0.$ Let $\bar{\alpha}^{c},\bar{\beta}^{c}:[0;+\infty)\rightarrow\mathbb{R}$
be such two step functions that 
\[
\left\Vert \alpha-\bar{\alpha}^{c}\right\Vert _{\left[t_{0};+\infty\right)}\leq c/2\mbox{ and }\left\Vert \beta-\bar{\beta}^{c}\right\Vert _{\left[t_{0};+\infty\right)}\leq c/2
\]
and set $\alpha^{c}=\bar{\alpha}^{c}-3c,$ $\beta^{c}=\bar{\beta}^{c}+3c.$
We naturally have $\alpha^c \leq \beta^c$ and
\[
\left\Vert \left(\alpha+\beta\right)-\left(\alpha^{c}+\beta^{c}\right)\right\Vert _{\left[t_{0};+\infty\right)} = \left\Vert \left(\alpha+\beta\right)-\left(\bar{\alpha}^{c}+\bar{\beta}^{c}\right)\right\Vert _{\left[t_{0};+\infty\right)}\leq c.
\]
Define $\gamma=\alpha-\beta,$ $\gamma^{c}=\beta^{c}-\alpha^{c}.$
Notice that $\beta^{c}\geq\beta+2c,$ similarly $\alpha^{c}\leq\alpha-2c,$
thus $\gamma^{c}\geq\gamma+4c.$ We also find $\psi^{c}\in St[t_{0};+\infty)$
such that $\left\Vert \psi-\psi^{c}\right\Vert _{\left[t_{0};+\infty\right)}\leq c/2.$
Now we consider functions $[t_{0};+\infty)\ni t\mapsto UTV^{\alpha^{c},\beta^{c}}\left(\psi^{c},\left[t_{0};t\right]\right),$
$[t_{0};+\infty)\ni t\mapsto UTV^{\alpha^{c},\beta^{c}}\left(\psi^{c},\left[t_{0};t\right]\right)$
and $[t_{0};+\infty)\ni t\mapsto TV^{\alpha^{c},\beta^{c}}\left(\psi^{c},\left[t_{0};t\right]\right).$
Let $\tilde{\psi}=\psi-\frac{1}{2}\left(\alpha+\beta\right),$ $\tilde{\psi}^{c}=\psi^{c}-\frac{1}{2}\left(\alpha^{c}+\beta^{c}\right).$
Notice now that for any $t_{0}\leq s<t,$ $\left|\tilde{\psi}^{c}\left(t\right)-\tilde{\psi}^{c}\left(s\right)\right|\leq\left|\tilde{\psi}\left(t\right)-\tilde{\psi}\left(s\right)\right|+2c$
and 
\begin{align*}
\left|\tilde{\psi}^{c}\left(t\right)-\tilde{\psi}^{c}\left(s\right)\right|-\frac{1}{2}\left(\gamma^{c}\left(t\right)+\gamma^{c}\left(s\right)\right) & \leq\left|\tilde{\psi}\left(t\right)-\tilde{\psi}\left(s\right)\right|+2c-\frac{1}{2}\left(\gamma\left(t\right)+\gamma\left(s\right)+4c\right)\\
 & \leq\left|\tilde{\psi}\left(t\right)-\tilde{\psi}\left(s\right)\right|-\frac{1}{2}\left(\gamma\left(t\right)+\gamma\left(s\right)\right).
\end{align*}
Hence $TV^{\alpha^{c},\beta^{c}}\left(\psi^{c},\left[t_{0};t\right]\right)\leq TV^{\alpha,\beta}\left(\psi,\left[t_{0};t\right]\right).$
On the other hand, in the limit we get 
\[
\limsup_{c\downarrow0}TV^{\alpha^{c},\beta^{c}}\left(\psi^{c},\left[t_{0};t\right]\right)\geq TV^{\alpha,\beta}\left(\psi,\left[t_{0};t\right]\right),
\]
hence we have pointwise convergence of $[t_{0};+\infty)\ni t\mapsto TV^{\alpha^{c},\beta^{c}}\left(\psi^{c},\left[t_{0};t\right]\right)$
to $[t_{0};+\infty)\ni t\mapsto TV^{\alpha,\beta}\left(\psi,\left[t_{0};t\right]\right)$
as $c\downarrow0$ and similar convergences hold for $UTV$ and $DTV.$
By the results obtained in Step 1 there exist functions 
\[
u^{\alpha^{c},\beta^{c}}\left(\cdot\right)=u^{\alpha^{c},\beta^{c}}\left(0\right)+UTV^{\alpha^{c},\beta^{c}}\left(\psi^{c},\left[t_{0};\cdot\right]\right)-DTV^{\alpha^{c},\beta^{c}}\left(\psi^{c},\left[t_{0};\cdot\right]\right)
\]
such that $\alpha^{c}\leq\psi^{c}-\psi^{\alpha^{c},\beta^{c}}\leq\beta^{c}.$
Define $\psi^{0}\left(c\right):=\psi^{\alpha^{c},\beta^{c}}\left(0\right).$
From the just obtained pointwise convergence we get that taking a sequence $c_{n}\downarrow 0$ such that the sequence $\psi^{0}\left(c_{n}\right)$
converges to some number $\psi^{0}$ and defining 
\[
\psi^{\alpha,\beta}\left(\cdot\right)=\psi^{0}+UTV^{\alpha,\beta}\left(\psi,\left[t_{0};\cdot\right]\right)-DTV^{\alpha,\beta}\left(\psi,\left[t_{0};\cdot\right]\right)
\]
we get a function satisfying all assertions of Theorem \ref{cond_tv_alpha_beta}.
Thus, by the assertions obtained for step functions in Step 1, we
get analogous assertions for any regulated $\psi$ and regulated
characteristics $\alpha\leq\beta$ satisfying condition (\ref{eq:cond_finiteness}). 

\end{proof}

\begin{rmk}
From (\ref{eq:Jordan1-0}), (\ref{eq:Jordan1-1}) and the Jordan decomposition we get that for any $u, \alpha, \beta \in G[0;+\infty),$ with $\alpha \leq \beta$ and any $0\leq a <b,$
$$TV^{\alpha, \beta}\left(\psi,\left[a;b\right]\right) = UTV^{\alpha, \beta}\left(\psi,\left[a;b\right]\right) + DTV^{\alpha, \beta}\left(\psi,\left[a;b\right]\right).$$
\end{rmk}
\begin{rmk} \label{gamma} It is easy
to see that if $\gamma_{0}:=\inf_{t\geq0}\left(\beta\left(t\right)-\alpha\left(t\right)\right)>0$
for regulated boundaries $\alpha,$ $\beta,$ then conditions (\ref{eq:cond_finiteness})
and (\ref{eq:finiteness}) for $u,$ resp. $\psi$ hold. Indeed, recall
the well known fact that every regulated function may be uniformly
approximated with arbitrary accuracy by step functions. Assuming that
$\gamma_{0}>0,$ we approximate the function $\tilde{\psi}$ by a step
function $\tilde{\psi}^{0}$ with accuracy $\gamma_{0}/2.$ Since
$\tilde{\psi}^{0},$ as a step function, has locally finite variation
and since for any $t_{0}\leq s_{1}<s_{2}\leq t,$ 
\begin{align*}
 & \left|\tilde{\psi}\left(s_{2}\right)-\tilde{\psi}\left(s_{1}\right)\right|-\frac{1}{2}\left(\gamma\left(s_{1}\right)+\gamma\left(s_{2}\right)\right)\\
 & \leq\left|\tilde{\psi}^{0}\left(s_{2}\right)-\tilde{\psi}^{0}\left(s_{1}\right)\right|+\gamma_{0}-\frac{1}{2}\left(\gamma\left(s_{1}\right)+\gamma\left(s_{2}\right)\right)\leq\left|\tilde{\psi}^{0}\left(s_{2}\right)-\tilde{\psi}^{0}\left(s_{1}\right)\right|
\end{align*}
we obtain $TV^{\alpha,\beta}\left(\psi,\left[t_{0};t\right]\right)\leq TV\left(\tilde{\psi^{c}},\left[t_{0};t\right]\right)<+\infty.$ 

Moreover, from the proof of Theorem \ref{cond_tv_alpha_beta} one
may see that when $\gamma_{0}>0,$ then the function $\psi^{\alpha,\beta}$
of Theorem \ref{cond_tv_alpha_beta} is piecewise monotone. Indeed,
the function $\psi^{\alpha,\beta}$ is non-decreasing as long as it
does not touch the boundary $\psi-\alpha$ and it is non-increasing
as long as it does not touch the boundary $\psi-\beta.$ There are
finitely many ``switches'' of $\psi^{\alpha,\beta}$ between these
two boundaries, since every such a switch corresponds to the oscillation
of $\psi-\alpha$ or $\psi-\beta$ of magnitude no less than $\gamma_{0},$
which may take place only finitely many times since $\psi-\alpha$ and $\psi-\beta$ are regulated. \end{rmk} 
Now, we will see that
when condition (\ref{eq:cond_finiteness}) is satisfied and the input
and boundaries are regulated, then the play operator may be defined
with upward and downward truncated variations. Moreover, this is the
only case (for regulated input and regulated boundaries), when the
output has locally finite total variation. We have the following \begin{cor}\label{Play_UTVDTV-1}
Let $\alpha,\beta,u\in G[0;+\infty),$ $\alpha \leq \beta,$ $\xi^{0}\in\left[u\left(0\right)-\beta\left(0\right);u\left(0\right)-\alpha\left(0\right)\right]$
and assume that condition (\ref{eq:cond_finiteness}) holds. Define $\alpha^{0}\left(0\right)=0,$ $\beta^{0}\left(0\right)=0,$ $u^{\xi^{0}}\left(0\right)=\xi^{0}$ and for $t>0,$
$\alpha^{0}\left(t\right)=\alpha\left(t\right),$ $\beta^{0}\left(t\right)=\beta\left(t\right),$
$u^{\xi^{0}}\left(t\right)=u\left(t\right).$ 
Then, for any $t\geq0,$ $TV^{\alpha^{0},\beta^{0}}\left(u^{\xi^{0}},\left[0;t\right]\right)<+\infty,$
and the play operator $\mathfrak{p}_{\alpha,\beta}\left[\xi^{0},\cdot\right]$
may be defined as 
\begin{equation}
\mathfrak{p}_{\alpha,\beta}\left[\xi^{0},u\right]\left(t\right)=\xi^{0}+UTV^{\alpha^{0},\beta^{0}}\left(u^{\xi^{0}},\left[0;t\right]\right)-DTV^{\alpha^{0},\beta^{0}}\left(u^{\xi^{0}},\left[0;t\right]\right).\label{eq:Jordan33}
\end{equation}
Moreover 
\begin{equation}
UTV\left(\mathfrak{p}_{\alpha,\beta}\left[\xi^{0},u\right],\left[0;t\right]\right)=UTV^{\alpha^{0},\beta^{0}}\left(u^{\xi^{0}},\left[0;t\right]\right)\label{eq:Jordan44}
\end{equation}
and 
\begin{equation}
DTV\left(\mathfrak{p}_{\alpha,\beta}\left[\xi^{0},u\right],\left[0;t\right]\right)=DTV^{\alpha^{0},\beta^{0}}\left(u^{\xi^{0}},\left[0;t\right]\right).\label{eq:Jordan55}
\end{equation}
\end{cor} \begin{proof} First, we prove that for any $t\geq0,$
\[
TV^{\alpha^{0},\beta^{0}}\left(u^{\xi^{0}},\left[0;t\right]\right)<+\infty.
\]
Indeed, by Theorem \ref{cond_tv_alpha_beta}, there exists such $u^{\alpha,\beta}\in G[0;+\infty)$
that $\alpha\leq u-u^{\alpha,\beta}\leq\beta$ and for $t\geq0$ 
\[
TV\left(u^{\alpha,\beta},\left[0;t\right]\right)=TV^{\alpha,\beta}\left(u,\left[0;t\right]\right)<+\infty.
\]
Notice now, that defining 
\[
u^{\alpha,\beta,\xi^{0}}\left(t\right)=\begin{cases}
\xi^{0} & \mbox{for }t=0;\\
u^{\alpha,\beta}\left(t\right) & \mbox{ for }t>0,
\end{cases}
\]
we get a function $u^{\alpha,\beta,\xi^{0}}\in G[0;+\infty)$ with
locally finite total variation and such that $\alpha^{0}\leq u^{\xi^{0}}-u^{\alpha,\beta,\xi^{0}}\leq\beta^{0}.$
By the minimality of $TV^{\alpha^{0},\beta^{0}}\left(u^{\xi^{0}},\left[0;t\right]\right)$
we get 
\begin{equation}
TV^{\alpha^{0},\beta^{0}}\left(u^{\xi^{0}},\left[0;t\right]\right)\leq TV\left(u^{\alpha,\beta,\xi^{0}},\left[0;t\right]\right)<+\infty.\label{eq:gwiazdka}
\end{equation}
Now we know that condition (\ref{eq:cond_finiteness}) for $\psi=u^{\xi^{0}}$
and boundaries $\alpha^{0},\beta^{0}$ hold. Applying Theorem
\ref{cond_tv_alpha_beta} we obtain such a function $u^{\alpha^{0},\beta^{0}}$
that $\alpha^{0}\leq u^{\xi^{0}}-u^{\alpha^{0},\beta^{0}}\leq\beta^{0}$
and for $t\geq0$ 
\begin{equation}
TV\left(u^{\alpha^{0},\beta^{0}},\left[0;t\right]\right)=TV^{\alpha^{0},\beta^{0}}\left(u^{\xi^{0}},\left[0;t\right]\right).\label{eq:ala}
\end{equation}
We have that $u^{\alpha^{0},\beta^{0}}\left(0\right)=\xi^{0}$ and
if there existed some other function $\xi$ such that $\xi\left(0\right)=\xi^{0}$
and $\alpha\leq u-\xi\leq\beta$ with smaller total variation than $u^{\alpha^{0},\beta^{0}}$ on some
interval $\left[0;t_{0}\right],$ $t_{0}\geq0,$ we would have 
\[
TV^{\alpha^{0},\beta^{0}}\left(u^{\xi^{0}},\left[0;t_{0}\right]\right)\leq TV\left(\xi,\left[0;t_{0}\right]\right)<TV\left(u^{\alpha^{0},\beta^{0}},\left[0;t_{0}\right]\right)
\]
(since $\xi$ also satisfies $\alpha^{0}\leq u^{\xi^{0}}-\xi\leq\beta^{0}$), but this is a contradiction with (\ref{eq:ala}). From this and the second part of Theorem \ref{cond_tv_alpha_beta} (applied to $u^{\xi^{0}},$ $\alpha^0$ and $\beta^0$) it follows that $u^{\alpha^{0},\beta^{0}}$
is the unique function with starting value $\xi^{0},$ satisfying
$\alpha\leq u-u^{\alpha^{0},\beta^{0}}\leq\beta$ and with the smallest
variations on the intervals $\left[0;t\right],$ $t\geq0.$ Thus,
it is the output of the play operator with input $u,$ starting condition
$\xi^{0}$ and and boundaries (characteristics) $\alpha,$ $\beta.$
By (\ref{eq:Jordan1-1}) and (\ref{eq:Jordan-1}) from Theorem \ref{cond_tv_alpha_beta} for $u^{\alpha^{0},\beta^{0}}$
we also have equalities (\ref{eq:Jordan33})-(\ref{eq:Jordan55}).\end{proof}
A special case of Corollary \ref{Play_UTVDTV-1} is \begin{cor}\label{Play_UTVDTV}
Let $c>0,$ $u\in G[0;+\infty)$ and $\xi^{0}\in\left[u\left(0\right)-c/2;u\left(0\right)+c/2\right].$
Define $u^{\xi^{0}}\left(0\right)=\xi^{0},$ $c^{0}\left(0\right)=0$ and for $t>0$ define $u^{\xi^{0}}\left(t\right)=u\left(t\right),$ $c^{0}\left(t\right)=c/2.$
Then for any $t\geq0$ the play operator $\mathfrak{p}_{c/2}\left[\xi^{0},\cdot\right]$
may be defined as 
\begin{equation}
\mathfrak{p}_{c/2}\left[\xi^{0},u\right]\left(t\right)=\xi^{0}+UTV^{-c^{0},c^{0}}\left(u^{\xi^{0}},\left[0;t\right]\right)-DTV^{-c^{0},c^{0}}\left(u^{\xi^{0}},\left[0;t\right]\right),\label{eq:jordan66}
\end{equation}
moreover 
\begin{equation}
UTV\left(\mathfrak{p}_{c/2}\left[\xi^{0},u\right],\left[0;t\right]\right)=UTV^{-c^{0},c^{0}}\left(u^{\xi^{0}},\left[0;t\right]\right)\label{eq:jordan77}
\end{equation}
and 
\begin{equation}
DTV\left(\mathfrak{p}_{c/2}\left[\xi^{0},u\right],\left[0;t\right]\right)=UTV^{-c^{0},c^{0}}\left(u^{\xi^{0}},\left[0;t\right]\right).\label{eq:jordan88}
\end{equation}
\end{cor}
\begin{rmk}
It is easy to see that we have the following formulas 
\begin{eqnarray}
 UTV^{-c^{0},c^{0}}\left(u^{\xi^{0}},\left[a;b\right]\right)\label{eq:formula1}
 =\sup_{n}\sup_{a<t_{1}<\cdots<t_{n}\leq b}\left\{ \left(u\left(t_{1}\right)-\xi^{0}-\frac{1}{2}c\right)_{+}+\sum_{i=2}^{n}\left(u\left(t_{i}\right)-u\left(t_{i-1}\right)-c\right)_{+}\right\} ,\end{eqnarray}
 
\begin{eqnarray}
 DTV^{-c^{0},c^{0}}\left(u^{\xi^{0}},\left[a;b\right]\right)\label{eq:formula2}
 =\sup_{n}\sup_{a<t_{1}<\cdots<t_{n}\leq b}\left\{ \left(\xi^{0}-u\left(t_{1}\right)+\frac{1}{2}c\right)_{+}+\sum_{i=2}^{n}\left(u\left(t_{i-1}\right)-u\left(t_{i}\right)-c\right)_{+}\right\} .
\end{eqnarray}
\end{rmk}
\begin{rmk}\label{bounds} We have the following
bounds for the total variation of the output of the play operator.
For any $t\geq0$ 
\[
TV^{\alpha,\beta}\left(u,\left[0;t\right]\right)\leq TV\left(\mathfrak{p}_{\alpha,\beta}\left[\xi^{0},u\right],\left[0;t\right]\right)\leq TV^{\alpha,\beta}\left(u,\left[0;t\right]\right)+\beta\left(0\right)-\alpha\left(0\right).
\]
The lower bound follows immediately from the fact that $\alpha\leq u-\mathfrak{p}_{\alpha,\beta}\left[\xi^{0},u\right]\leq\beta$
and the upper bound is simply the estimate of the total variation
of the function $u^{\alpha,\beta,\xi^{0}}$ defined in the proof of
Corollary \ref{Play_UTVDTV-1} and follows from (\ref{eq:gwiazdka})
since for $u^{\alpha,\beta,\xi^{0}}$ we have 
\begin{align*}
TV\left(u^{\alpha,\beta,\xi^{0}},\left[0;t\right]\right) & \leq\left|u^{\alpha,\beta,\xi^{0}}\left(0\right)-u^{\alpha,\beta}\left(0\right)\right|+TV\left(u^{\alpha,\beta},\left[0;t\right]\right)\\
 & \leq\beta\left(0\right)-\alpha\left(0\right)+TV^{\alpha,\beta}\left(u,\left[0;t\right]\right).
\end{align*}
Similar bounds hold for $UTV$ and $DTV.$ \end{rmk} 
In view of Corollaries
\ref{Play_UTVDTV-1}, \ref{Play_UTVDTV} and Definitions \ref{play_def-1},
\ref{play_def}, the play operator may be seen as the special case
of the solution of the following minimal variation problem: for given $\psi,  \alpha,\beta\in G[0;+\infty),$ such that $\alpha \leq \beta,$ find a function $\psi^{\alpha, \beta}$ such that $\alpha \leq \psi - \psi^{\alpha, \beta} \leq \beta $ with the minimal total variation possible on the intervals
$\left[t_{0};t\right],$ $t\geq t_{0},$ existence of which is guaranteed
by Theorem \ref{cond_tv_alpha_beta}. Unfortunately, when $\alpha\left(t_{0}\right)<\beta\left(t_{0}\right)$
the starting point of the solution is not specified in advance and
one has to look ``forward'' to find the optimal starting point.
This way we lose the semigroup property of the play operator: 
for $t>t_{0},$ 
\[
\mathfrak{p}_{\alpha,\beta}\left[\xi^{0},u\right]\left(t\right)=\mathfrak{p}_{\alpha,\beta}\left[\mathfrak{p}_{\alpha,\beta}\left[\xi^{0},u\right]\left(t_{0}\right),u\left(t_{0}+\cdot\right)\right]\left(t-t_{0}\right).
\]
Moreover, the solution of the minimal variation problem may not be unique (this may happen only when the solution is a constant function).
\begin{rmk} From the proof of Theorem \ref{cond_tv_alpha_beta} it
is possible to infer that the optimal starting point of $\psi^{\alpha,\beta}$
is 
\[
\psi^{\alpha,\beta}\left(t_{0}\right)=\begin{cases}
\inf_{t\in\left[t_{0};T_{U}\psi\right]}\left(\psi\left(t\right)-\alpha\left(t\right)\right)\mbox{ if }T_{U}\psi<T_{D}\psi;\\
\sup_{t\in\left[t_{0};T_{U}\psi\right]}\left(\psi\left(t\right)-\beta\left(t\right)\right)\mbox{ if }T_{U}\psi>T_{D}\psi,
\end{cases}
\]
where
\[
T_{U}\psi=\inf\left\{ t\geq t_{0}:\left(\psi-\beta\right)\left(t\right)>\inf_{s\in\left[t_{0};t\right]}\left(\psi-\alpha\right)\left(s\right)\right\} ,
\]
\[
T_{D}\psi=\inf\left\{ t\geq t_{0}:\left(\psi-\alpha\right)\left(t\right)<\sup_{s\in\left[t_{0};t\right]}\left(\psi-\beta\right)\left(s\right)\right\} .
\]
\end{rmk}

\section{The generalised play operator }

When the condition (\ref{eq:cond_finiteness}) is not satisfied, we
naturally can not apply the truncated variation formalism to define the
play operator. However, it is possible to extend the definition of
the play operator to cases when (\ref{eq:cond_finiteness}) does not
hold. First of such extensions, in a slightly different setting, based on differential inclusions, which may not coincide with ours, was given in \cite[Chapt. III, Theorem 2.2]{V}, compare also \cite[Chapt. III, formula (2.13)]{V} with formula (\ref{eq:recursion}). Next, the extension for any c\`{a}dl\`{a}g input and (symmetric) c\`{a}dl\`{a}g boundary conditions,  in similar setting as ours, is given e.g. by \cite[Corollary 2.3]{K1} (see
also \cite[Theorem 4.1]{R2008}). All these approaches utilize the Lipschitz continuity
of the play operator with respect to the inputs, boundaries (characteristics),
in sup norm, and starting conditions. The symmetry of the boundaries
is no restrictive, since for non-symmetric boundaries $\alpha\leq\beta,$
one may symmetrize the input $u,$ defining $\tilde{u}=u-\frac{1}{2}\left(\alpha+\beta\right),$
and apply to $\tilde{u}$ the play operator with symmetric characteristics
$-\frac{1}{2}\left(\beta-\alpha\right),$ $\frac{1}{2}\left(\beta-\alpha\right)$.
Similar extension for c\`{a}dl\`{a}g input and c\`{a}dl\`{a}g boundary conditions but obtained with a different definition, is given by \cite[Definition 2.2]{BurdzySPA}.
In \cite[Definition 2.2]{BurdzySPA} the starting condition $\xi^{0}$
is not specified, but it may be
imposed with the shift of the input - in \cite[Definition 2.2]{BurdzySPA}  one substitutes  $\psi$ with 
$\psi-\psi\left(0\right)+\xi^{0}.$
Some condition on characteristics guaranteeing that the output will
have infinite variation, when the input is a typical path of a standard
Brownian motion, is given by \cite[Theorem 4.3]{BurdzySPA}. 

Another extensions, for much more general inputs - $L^{\infty}$ or
Lebesgue measurable inputs, but only for constant, symmetric boundaries
- were given in \cite{KL1} and \cite{IE}. 

Below we extend slightly the definitions of \cite{K1}, \cite{R2008}
to any regulated input $u:\left[0;+\infty\right)\rightarrow\mathbb{R},$
regulated boundaries (characteristics) $\alpha,\mbox{ }\beta:\left[0;+\infty\right)\rightarrow\mathbb{R},$
with $\alpha\leq\beta$ and starting condition $\xi^{0}\in\left[u\left(0\right)-\beta\left(0\right);u\left(0\right)-\alpha\left(0\right)\right].$
We apply similar approximation techniques as in \cite[Corollary 2.3]{K1},
\cite[Theorem 4.1]{R2008}, i.e. we will utilize the fact that every
regulated function $\chi:\left[0;+\infty\right)\rightarrow\mathbb{R}$
may be uniformly approximated by step functions of the form 
\[
h\left(t\right)=\sum_{k=0}^{\infty}h_{2k}1_{\left\{ t_{2k}\right\} }\left(t\right)+\sum_{k=0}^{\infty}h_{2k+1}1_{\left(t_{2k};t_{2k+1}\right)}\left(t\right),
\]
where $0=t_{0}<t_{1}<\ldots$ and $\lim_{k\rightarrow+\infty}t_{k}=+\infty.$
The set of such functions will be denoted by $St\left[0;+\infty\right).$
Now, for any characteristics $\alpha,$ $\beta\in St\left[0;+\infty\right),$
with $\alpha\leq\beta$ and step input $u$ the value of the output
of the play operator is the composition of the following ``one-step
operators'': 
\begin{equation} \label{eq:recursion}
\mathfrak{p}^{i}\left[\xi^{i-1}\right]:=\min\left\{ \max\left\{ u^{i}-\beta^{i},\xi^{i-1}\right\} ,u^{i}-\alpha^{i}\right\}  = \begin{cases}
u^i - \alpha^i & \mbox{if } \xi^{i-1} > u^i - \alpha^i;\\
\xi^{i-1} & \mbox{if } \xi^{i-1} \in \left[u^i - \beta^i ; u^i - \alpha^i \right];\\
 u^i - \beta^i & \mbox{if }\xi^{i-1} < u^i - \beta^i,
\end{cases}
\end{equation}
where 
\[
\alpha=\sum_{k=0}^{\infty}\alpha^{2k}1_{\left\{ t_{k}\right\} }\left(\cdot\right)+\sum_{k=0}^{\infty}\alpha^{2k+1}1_{\left(t_{k};t_{k+1}\right)}\left(\cdot\right)
\]
and similar representation (with the same $t_{i},$ $i=0,1,\ldots$)
holds for $\beta$ and $u.$ More precisely, for $t=t_{0},t_{1},\ldots$
we have $\mathfrak{p}_{\alpha,\beta}\left[\xi^{0},u\right]\left(t_{k}\right)=\xi^{2k}$
and for $t\in\left(t_{k};t_{k+1}\right)$ we have $\mathfrak{p}_{\alpha,\beta}\left[\xi^{0},u\right]\left(t\right)=\xi^{2k+1},$
$k=0,1,\ldots,$ where $\xi^{0}$ is given and $\xi^{i},$ $i=1,2\ldots,$
are defined recursively with (\ref{eq:recursion}), i.e. $\xi^{i}:=\mathfrak{p}^{i}\left[\xi^{i-1}\right].$ This algorithm produces such $\xi^i$ that the difference $\xi^{i} - \xi^{i-1}$ is the smallest possible and the relations $\alpha^i \leq u^i - \xi^i \leq \beta^i$ hold.
It is not difficult to observe that for a step input and step characteristics
the just defined operator $\mathfrak{p}_{\alpha,\beta}\left[\xi^{0},\cdot\right]$
coincides with the play operator of Definition \ref{play_def-1}.
\begin{rmk}
The fact that for step functions both definitions coincide follows easily e.g. by Theorem \ref{cond_tv_alpha_beta} and by the induction with respect to the number of steps. 
Unfortunately, it is not the case in higher dimensions. It is possible to construct an example of a step function attaining its values in the space $\mathbb{R}^2$ with the Euclidean metric, showing that even for constant boundary being the unit ball $B_2 := \left\{(x,y)\in \mathbb{R}^2: x^2 + y^2 \leq 1 \right\}$ the definition based on the counterpart to the greedy recursion (\ref{eq:recursion}) does not lead to the output with the smallest total variation possible on every interval $[0;t],$ $t>0.$ For example, for $u(t)=(2,0)1_{\left\{0\right\}}(t)+(0,0)1_{(0;1]}(t)+(0,2)1_{(1;+\infty)}(t),$ $\xi_0 = (2,0)$ such an algorithm would produce the output $\xi(t) = (2,0)1_{\left\{0\right\}}(t)+(1,0)1_{(0;1]}(t)+(1/\sqrt{5},2-2/\sqrt{5})1_{(1;+\infty)}(t) $ while the function $\chi (t) = (2,0)1_{\left\{0\right\}}(t)+(3/5,4/5)1_{(0;1]}(t)+(1/\sqrt{5},2-2/\sqrt{5})1_{(1;+\infty)}(t) $ has smaller variation on the interval $[0;2]$ and $\chi(0)=\xi_0,$ $\chi - u \in B_2.$
\end{rmk}

Now we will prove that for every two one-step operators 
\[
\mathfrak{p}_{1}\left[x\right]=\min\left\{ \max\left\{ \eta_{1},x\right\} ,\theta_{1}\right\} \mbox{ and }\mathfrak{p}_{2}\left[x\right]=\min\left\{ \max\left\{ \eta_{2},x\right\} ,\theta_{2}\right\} 
\]
we have 
\begin{equation}
\left|\mathfrak{p}_{1}\left[x_{1}\right]-\mathfrak{p}_{2}\left[x_{2}\right]\right|\leq\max\left\{ \left|x_{1}-x_{2}\right|,\left|\eta_{1}-\eta_{2}\right|,\left|\theta_{1}-\theta_{2}\right|\right\} .\label{eq:inequality}
\end{equation}
Though the inequality (\ref{eq:inequality}) is elementary, for reader's
convenience we present its proof.

\begin{proof} For any $x,\mbox{ }y,\mbox{ }z,\mbox{ }t\in\mathbb{R}$
we have 
\[
\min\left\{ z,t\right\} +\max\left\{ \left|x-z\right|,\left|y-t\right|\right\} \geq\min\left\{ z+\left|x-z\right|,t+\left|y-t\right|\right\} \geq\min\left\{ x,y\right\} 
\]
and 
\[
\min\left\{ x,y\right\} +\max\left\{ \left|x-z\right|,\left|y-t\right|\right\} \geq\min\left\{ x+\left|z-x\right|,y+\left|t-y\right|\right\} \geq\min\left\{ z,t\right\} ,
\]
hence 
\[
\max\left\{ \left|x-z\right|,\left|y-t\right|\right\} \geq\left|\min\left\{ x,y\right\} -\min\left\{ z,t\right\} \right|.
\]
Similarly, 
\[
\max\left\{ \left|x-z\right|,\left|y-t\right|\right\} \geq\left|\min\left\{ -x,-y\right\} -\min\left\{ -z,-t\right\} \right|=\left|\max\left\{ x,y\right\} -\max\left\{ z,t\right\} \right|
\]
Now, applying these inequalities we get 
\begin{align*}
\left|\mathfrak{p}_{1}\left[x_{1}\right]-\mathfrak{p}_{2}\left[x_{2}\right]\right| & =\left|\min\left\{ \max\left\{ \eta_{1},x_{1}\right\} ,\theta_{1}\right\} -\min\left\{ \max\left\{ \eta_{2},x_{1}\right\} ,\theta_{2}\right\} \right|\\
 & \leq\max\left\{ \left|\max\left\{ \eta_{1},x\right\} -\max\left\{ \eta_{2},x\right\} \right|,\left|\theta_{1}-\theta_{2}\right|\right\} \\
 & \leq\max\left\{ \max\left\{ \left|\eta_{1}-\eta_{2}\right|,\left|x_{1}-x_{2}\right|\right\} ,\left|\theta_{1}-\theta_{2}\right|\right\} \\
 & =\max\left\{ \left|x_{1}-x_{2}\right|,\left|\eta_{1}-\eta_{2}\right|,\left|\theta_{1}-\theta_{2}\right|\right\} .
\end{align*}

\end{proof} 

Now, by (\ref{eq:inequality}) and by recursion (\ref{eq:recursion})
we get 
\begin{align}
 & \left|\mathfrak{p}_{\alpha_{1},\beta_{1}}\left[\xi_{1}^{0},u_{1}\right]\left(t\right)-\mathfrak{p}_{\alpha_{2},\beta_{2}}\left[\xi_{2}^{0},u_{2}\right]\left(t\right)\right|\nonumber \\
 & \leq\max\left\{ \left|\xi_{1}^{0}-\xi_{2}^{0}\right|,\left\Vert \left(\beta_{1}-u_{1}\right)-\left(\beta_{2}-u_{2}\right)\right\Vert _{\left[0;+\infty\right)},\left\Vert \left(\alpha_{1}-u_{1}\right)-\left(\alpha_{2}-u_{2}\right)\right\Vert _{\left[0;+\infty\right)}\right\} \nonumber \\
 & \leq\max\left\{ \left|\xi_{1}^{0}-\xi_{2}^{0}\right|,\max\left\{ \left\Vert \alpha_{1}-\alpha_{2}\right\Vert _{\left[0;+\infty\right)}+\left\Vert \beta_{1}-\beta_{2}\right\Vert _{\left[0;+\infty\right)}\right\} +\left\Vert u_{1}-u_{2}\right\Vert _{\left[0;+\infty\right)}\right\} .\label{eq:inequality1}
\end{align}
By (\ref{eq:inequality1}) and the usual approximation argument (cf.
proof of \cite[Theorem 4.1]{R2008}), the play operator may be extended
to any regulated input $u$ and any regulated characteristics $\alpha,$
$\beta,$ with $\alpha\leq\beta.$ Thus we formulate 

\begin{defn} \label{play_def-1-1} Let $\alpha,\beta,u\in G\left[0;+\infty\right),$ $\alpha \leq \beta,$
$\xi^{0}\in$$\left[u\left(0\right)-\beta\left(0\right);u\left(0\right)-\alpha\left(0\right)\right].$
For any sequences $\alpha_{n},\beta_{n},u_{n}\in St[0;+\infty)$ and
$\xi_{n}^{0}\in\mathbb{R}$ such that $\alpha_n \leq \beta_n,$
$\xi^{0}_n\in$$\left[u_n(0)-\beta_n(0);u_n(0)-\alpha_n(0)\right]$ and 
\begin{equation}
\left\Vert u_{n}-u\right\Vert _{\left[0;+\infty\right)}+\left\Vert \alpha_{n}-\alpha\right\Vert _{\left[0;+\infty\right)}+\left\Vert \beta_{n}-\beta\right\Vert _{\left[0;+\infty\right)}+\left|\xi_{n}^{0}-\xi^{0}\right|\rightarrow0\label{eq:uniform}
\end{equation}
the limit $\mathfrak{p}_{\alpha_{n},\beta_{n}}\left[\xi_{n}^{0},u_{n}\right]$
exists (in sup norm $\left\Vert \cdot\right\Vert _{\left[0;+\infty\right)}$)
and is unique. Hence, we define the \emph{generalised play operator}
$\tilde{\mathfrak{p}}_{\alpha,\beta}\left[\xi^{0},\cdot\right]$ with
the following formula 
\[
\tilde{\mathfrak{p}}_{\alpha,\beta}\left[\xi^{0},u\right]\left(t\right):=\lim_{n\rightarrow+\infty}\mathfrak{p}_{\alpha_{n},\beta_{n}}\left[\xi_{n}^{0},u_{n}\right]\left(t\right).
\]

\end{defn} 

Of course, we did not proved that the just defined extension coincides
with the play operator of Definition \ref{play_def-1} for any regulated
input $u$ and any regulated characteristics $\alpha,$ $\beta,$
$\alpha\leq\beta,$ satisfying condition (\ref{eq:cond_finiteness}).
This is the subject of the following

\begin{lem} For $u\in G[0;+\infty),$ $\xi^{0}\in\mathbb{R}$ and
characteristics $\alpha,\beta\in G[0;+\infty),$ such that $\alpha\leq\beta,$
$\xi^{0}\in\left[u\left(0\right)-\beta\left(0\right);u\left(0\right)-\alpha\left(0\right)\right]$
and condition (\ref{eq:cond_finiteness}) holds, we have 
\[
\tilde{\mathfrak{p}}_{\alpha,\beta}\left[\xi^{0},u\right]=\mathfrak{p}_{\alpha,\beta}\left[\xi^{0},u\right].
\]
\end{lem}\begin{proof} Let us fix $c>0$ and consider function $u^{\xi^{0}}$
and boundaries $\alpha^{0},$ $\beta^{0}$ defined in Corollary \ref{Play_UTVDTV-1}.
It is possible to approximate $u^{\xi^{0}},$ $\alpha^{0}$ and $\beta^{0}$
uniformly with accuracy $4c$ by step functions $u^{c,\xi^{0}},$
$\alpha^{c}$ and $\beta^{c}$ such that $\alpha^{c} \leq \beta^{c}$ and we have pointwise convergences
$UTV^{\alpha^{c},\beta^{c}}\left(u^{c,\xi^{0}},\left[0;\cdot\right]\right)\rightarrow UTV^{\alpha^{0},\beta^{0}}\left(u^{\xi^{0}},\left[0;\cdot\right]\right),$
$DTV^{\alpha^{c},\beta^{c}}\left(u^{c,\xi^{0}},\left[0;\cdot\right]\right)\rightarrow DTV^{\alpha^{0},\beta^{0}}\left(u^{\xi^{0}},\left[0;\cdot\right]\right)$
as $c\downarrow0.$ The detailed construction is given in Step 2 of
the proof of Theorem \ref{cond_tv_alpha_beta}. Thus,
by Corollary \ref{Play_UTVDTV-1} we get the pointwise convergence 
\begin{align*}
\mathfrak{p}_{\alpha^{c},\beta^{c}}\left[\xi^{0},u^{c,\xi^{0}}\right] & =\xi^{0}+UTV^{\alpha^{c},\beta^{c}}\left(u^{c,\xi^{0}},\left[0;\cdot\right]\right)-DTV^{\alpha^{c},\beta^{c}}\left(u^{c,\xi^{0}},\left[0;\cdot\right]\right)\\
 & \rightarrow\mathfrak{p}_{\alpha,\beta}\left[\xi^{0},u\right]
\end{align*}
as $c\downarrow0.$ On the other hand, from Definition \ref{play_def-1}
it is not difficult to see that for step functions $u^{c,\xi^{0}},$
$\alpha^{c}$ and $\beta^{c}$ we have 
\[
\mathfrak{p}_{\alpha^{c},\beta^{c}}\left[\xi^{0},u^{c,\xi^{0}}\right]=\mathfrak{\tilde{p}}_{\alpha^{c},\beta^{c}}\left[\xi^{0},u^{c,\xi^{0}}\right].
\]
Now, by the just established estimate (\ref{eq:inequality1}) for
step functions we know that the uniform convergence on the halfline $[0;+\infty),$ 
\[
\tilde{\mathfrak{p}}_{\alpha^{c},\beta^{c}}\left[\xi^{0},u^{c}\right]\rightrightarrows\tilde{\mathfrak{p}}_{\alpha,\beta}\left[\xi^{0},u\right],
\]
holds as $c\downarrow0.$ Thus, both operators coincide for $u\in G[0;+\infty),$
$\xi^{0}\in\mathbb{R}$ and characteristics $\alpha,$ $\beta\in G[0;+\infty),$
such that $\alpha\leq\beta,$ $\xi^{0}\in\left[u\left(0\right)-\beta\left(0\right);u\left(0\right)-\alpha\left(0\right)\right]$
and condition (\ref{eq:cond_finiteness}) holds.\end{proof}

From inequality (\ref{eq:inequality1}) and Definition \ref{play_def-1-1}
we naturally obtain the Lipschitz continuity of the generalised play
operator, i.e. we have \begin{lem} For $u_{1},u_{2}\in G[0;+\infty),$
$\xi_{1}^{0},\xi_{2}^{0}\in\mathbb{R}$ and characteristics $\alpha_{1},\alpha_{2},\beta_{1},\beta_{2}\in G[0;+\infty),$
such that $\alpha_{1}\leq\beta_{1},$ $\alpha_{2}\leq\beta_{2},$
$\xi_{1}^{0}\in\left[u_{1}\left(0\right)-\beta_{1}\left(0\right);u_{1}\left(0\right)-\alpha_{1}\left(0\right)\right],$
$\xi_{2}^{0}\in\left[u_{2}\left(0\right)-\beta_{2}\left(0\right);u_{2}\left(0\right)-\alpha_{2}\left(0\right)\right]$
we have 
\begin{align*}
 & \left\Vert \tilde{\mathfrak{p}}_{\alpha_{1},\beta_{1}}\left[\xi_{1}^{0},u_{1}\right]-\tilde{\mathfrak{p}}_{\alpha_{2},\beta_{2}}\left[\xi_{2}^{0},u_{2}\right]\right\Vert _{\left[0;+\infty\right)}\\
 & \leq\max\left\{ \left|\xi_{1}^{0}-\xi_{2}^{0}\right|,\max\left\{ \left\Vert \alpha_{1}-\alpha_{2}\right\Vert _{\left[0;+\infty\right)}+\left\Vert \beta_{1}-\beta_{2}\right\Vert _{\left[0;+\infty\right)}\right\} +\left\Vert u_{1}-u_{2}\right\Vert _{\left[0;+\infty\right)}\right\} .
\end{align*}
\end{lem} 
Unfortunately, even when the condition (\ref{eq:cond_finiteness})
is satisfied we shall not expect a similar result for the total
variations of the output, i.e. that for any sequences $\alpha_{n},\beta_{n},u_{n}\in G[0;+\infty)$
and $\xi_{n}^{0}\in\mathbb{R}$ satisfying conditions appearing in Definition \ref{play_def-1-1} we will have convergence, or at least a local uniform bound of the form (cf. \cite[Lemma 4.4]{KL2}): for every $t>0$ there exist such $C\left(t\right)$
that 
\[
\sup_{n\geq1}TV\left(\mathfrak{p}_{\alpha_{n},\beta_{n}}\left[\xi_{n}^{0},u_{n}\right],\left[0;t\right]\right)\leq C\left(t\right).
\]
However, in the one-dimensional case when $\gamma_{0}>0$ (the definition
of $\gamma_{0}$ is given in Remark \ref{gamma}) we have the following,
more precise statement than \cite[Lemma 4.4]{KL2}: 
\begin{lem}
Let $\alpha,\beta,u\in G\left[0;+\infty\right),$ $\xi^{0}\in$ $\left[u\left(0\right)-\beta\left(0\right);u\left(0\right)-\alpha\left(0\right)\right]$
and assume that $\inf_{t\geq0}\left\{ \beta\left(t\right)-\alpha\left(t\right)\right\} >0.$
For any sequences $\alpha_{n},\beta_{n},u_{n}\in G[0;+\infty)$ and
$\xi_{n}^{0}\in\mathbb{R}$ satisfying $\alpha_n \leq \beta_n,$
$\xi^{0}_n\in$ $\left[u_n(0)-\beta_n(0);u_n(0)-\alpha_n(0)\right]$ and (\ref{eq:uniform}), functions
$TV^{\alpha_{n},\beta_{n}}\left(u_{n},\left[0;\cdot\right]\right),$
$UTV^{\alpha_{n},\beta_{n}}\left(u_{n},\left[0;\cdot\right]\right)$
and $DTV^{\alpha_{n},\beta_{n}}\left(u_{n},\left[0;\cdot\right]\right)$
converge uniformly on compact subsets of $\left[0;+\infty\right)$
to $TV^{\alpha,\beta}\left(u,\left[0;\cdot\right]\right),$ $UTV^{\alpha,\beta}\left(u,\left[0;\cdot\right]\right)$
and $DTV^{\alpha,\beta}\left(u,\left[0;\cdot\right]\right)$ respectively.
Similar convergences hold for total, positive and negative variations
of the outputs of the play operators $\mathfrak{p}_{\alpha_{n},\beta_{n}}\left[\xi_{n}^{0},u_{n}\right].$
\end{lem} 
\begin{proof} Let us fix $T>0$ and notice that for any
$\delta>0,$ due to the regularity of $u,$ there exists a number
$N\left(T,\delta\right)$ such that for any $0\leq t_{0}<t_{1}<\ldots<t_{N}\leq t_{N+1}=T$
we have 
\begin{equation}
\mbox{if }\left|u\left(t_{i}\right)-u\left(t_{i-1}\right)\right|>\delta\mbox{ or }\left|u\left(t_{i}\right)-u\left(t_{i+1}\right)\right|>\delta\mbox{ for }i=1,2,\ldots,N\mbox{ then }N\leq N\left(T,\delta\right).\label{eq:uniform_osc}
\end{equation}
Also, for $\varepsilon>0$ and the sequences $\alpha_{n},\beta_{n},u_{n}\in G[0;+\infty)$
we may find $n\left(\varepsilon\right)$ such that for $n\geq n\left(\varepsilon\right)$
and any $0\leq s<t\leq T,$ 
\begin{equation}
\left|\tilde{u}_{n}\left(t\right)-\tilde{u}_{n}\left(s\right)\right|-\frac{1}{2}\left(\gamma_{n}\left(t\right)+\gamma_{n}\left(s\right)\right)\leq\left|\tilde{u}\left(t\right)-\tilde{u}\left(s\right)\right|-\frac{1}{2}\left(\gamma\left(t\right)+\gamma\left(s\right)\right)+\varepsilon,\label{eq:krowa1}
\end{equation}
\begin{equation}
\left|\tilde{u}_{n}\left(t\right)-\tilde{u}_{n}\left(s\right)\right|-\frac{1}{2}\left(\gamma_{n}\left(t\right)+\gamma_{n}\left(s\right)\right)\geq\left|\tilde{u}\left(t\right)-\tilde{u}\left(s\right)\right|-\frac{1}{2}\left(\gamma\left(t\right)+\gamma\left(s\right)\right)-\varepsilon.\label{eq:krowa1-1}
\end{equation}
(Here $\tilde{u}=u-\frac{1}{2}\left(\alpha+\beta\right),$ $\gamma=\beta-\alpha,$
$\gamma_{0}=\inf_{t\geq0}\gamma\left(t\right),$ $\tilde{u}_{n}=u_{n}-\frac{1}{2}\left(\alpha_{n}+\beta_{n}\right),$
$\gamma_{n}=\beta_{n}-\alpha_{n}.$) Let $\varepsilon<\gamma_{0}/2.$
For $n\geq n\left(\varepsilon\right),$ $t\in\left(0;T\right]$ let
us consider a partition $0\leq t_{0}<t_{1}<\ldots<t_{N}\leq t$ such
that 
\begin{equation}
TV^{\alpha_{n},\beta_{n}}\left(u_{n},\left[0;t\right]\right)\leq\sum_{i=1}^{N}\left(\left|\tilde{u}_{n}\left(t_{i}\right)-\tilde{u}_{n}\left(t_{i-1}\right)\right|-\frac{1}{2}\left(\gamma_{n}\left(t_{i}\right)+\gamma_{n}\left(t_{i-1}\right)\right)\right)_{+}+\varepsilon.\label{eq:approx_tv}
\end{equation}
By (\ref{eq:krowa1}), for every $i=1,2,\ldots,N$ such that 
\[
\left|\tilde{u}_{n}\left(t_{i}\right)-\tilde{u}_{n}\left(t_{i-1}\right)\right|-\frac{1}{2}\left(\gamma_{n}\left(t_{i}\right)+\gamma_{n}\left(t_{i-1}\right)\right)>0
\]
we have that 
\[
\left|\tilde{u}\left(t_{i}\right)-\tilde{u}\left(t_{i-1}\right)\right|\geq\frac{1}{2}\left(\gamma\left(t_{i}\right)+\gamma\left(t_{i-1}\right)\right)-\varepsilon>\frac{1}{2}\gamma_{0}.
\]
Now, by (\ref{eq:uniform_osc}) we have at most $N\left(T,\gamma_{0}/2\right)$
summands in (\ref{eq:approx_tv}) which are non-zero and by (\ref{eq:krowa1})
we have 
\begin{align*}
TV^{\alpha_{n},\beta_{n}}\left(u_{n},\left[0;t\right]\right) & \leq\sum_{i=1}^{N}\left(\left|\tilde{u}\left(t_{i}\right)-\tilde{u}\left(t_{i-1}\right)\right|-\frac{1}{2}\left(\gamma_{n}\left(t_{i}\right)+\gamma_{n}\left(t_{i-1}\right)\right)\right)_{+}\\
 & +\left(N\left(T,\gamma_{0}/2\right)+1\right)\varepsilon\\
 & \leq TV^{\alpha,\beta}\left(u,\left[0;t\right]\right)+\left(N\left(T,\gamma_{0}/2\right)+1\right)\varepsilon.
\end{align*}
On the other hand, every sum approximating $TV^{\alpha,\beta}\left(u,\left[0;t\right]\right)$
with arbitrary accuracy has at most $N\left(T,\gamma_{0}\right)$
non-zero terms and each of these terms is approximated for $n\geq n\left(\varepsilon\right)$
by $\left|\tilde{u}_{n}\left(t_{i}\right)-\tilde{u}_{n}\left(t_{i-1}\right)\right|-\frac{1}{2}\left(\gamma_{n}\left(t_{i}\right)+\gamma_{n}\left(t_{i-1}\right)\right)$
with accuracy $\varepsilon.$ Thus, for $n\geq n\left(\varepsilon\right),$
\[
TV^{\alpha,\beta}\left(u,\left[0;t\right]\right)\leq TV^{\alpha_{n},\beta_{n}}\left(u_{n},\left[0;t\right]\right)+N\left(T,\gamma_{0}\right)\varepsilon.
\]
Letting $\varepsilon\downarrow0$ we get the claimed uniform convergence
of $TV^{\alpha_{n},\beta_{n}}\left(u_{n},\left[0;\cdot\right]\right)$
on compacts. Similarly we obtain the claimed convergences of $UTV^{\alpha_{n},\beta_{n}}\left(u_{n},\left[0;\cdot\right]\right)$
and $DTV^{\alpha_{n},\beta_{n}}\left(u_{n},\left[0;\cdot\right]\right).$

The convergences for total, positive and negative variations of the
outputs of the play operators $\mathfrak{p}_{\alpha_{n},\beta_{n}}\left[\xi_{n}^{0},u_{n}\right]$
are obtained in a similar way, using the representation of Corollary \ref{Play_UTVDTV-1}.\end{proof}

\section*{Appendix. Proof of Lemma \ref{lema1} }

\begin{proof} From the estimate: for $\delta\in\left(0;1\right),$
$c>0$ and $x,y\in\mathbb{R}$ 
\[
\left(\left|x+y\right|-c\right)_{+}\leq\left(\left|x\right|-\delta c\right)_{+}+\left(\left|y\right|-\left(1-\delta\right)c\right)_{+}
\]
we immediately get 
\begin{equation}
TV^{c}\left(\psi_{1}+\psi_{2},\left[0;t\right]\right)\leq TV^{\delta c}\left(\psi_{1},\left[0;t\right]\right)+TV^{\left(1-\delta\right)c}\left(\psi_{2},\left[0;t\right]\right)\label{eq:estt1}
\end{equation}
and 
\[
TV^{c}\left(\psi_{1}+\psi_{2}-\psi_{2},\left[0;t\right]\right)\leq TV^{\delta c}\left(\psi_{1}+\psi_{2},\left[0;t\right]\right)+TV^{\left(1-\delta\right)c}\left(-\psi_{2},\left[0;t\right]\right),
\]
which simplifies to 
\begin{equation}
TV^{\delta c}\left(\psi_{1}+\psi_{2},\left[0;t\right]\right)\geq TV^{c}\left(\psi_{1},\left[0;t\right]\right)-TV^{\left(1-\delta\right)c}\left(\psi_{2},\left[0;t\right]\right).\label{eq:estt2}
\end{equation}
Now, let $\theta$ be the index of the function $\varphi,$ which means that   $\lim_{c\downarrow0}\varphi\left(c\right)/\varphi\left(\delta c\right)=\delta^{-\theta}$ (c.f. \cite[p. 18, Theorem 1.4.3 and definitions above]{BGT}).
Since $\varphi$ is non-decreasing, we have $\theta\geq0.$ Now, from
(\ref{eq:estt1}) we get 
\[
\limsup_{c\downarrow0}\varphi\left(c\right)\cdot TV^{c}\left(\psi_{1}+\psi_{2},\left[0;t\right]\right)\leq\delta^{-\theta}\limsup_{c\downarrow0}\varphi\left(c\right)\cdot TV^{c}\left(\psi_{1},\left[0;t\right]\right)
\]
and from (\ref{eq:estt2}) we get 
\[
\liminf_{c\downarrow0}\varphi\left(c\right)\cdot TV^{c}\left(\psi_{1}+\psi_{2},\left[0;t\right]\right)\geq\delta^{\theta}\liminf_{c\downarrow0}\varphi\left(c\right)\cdot TV^{c}\left(\psi_{1},\left[0;t\right]\right).
\]
Setting $\delta$ arbitrary close to $1$ we get the result. 
\end{proof}

\bf{Acknowledgments} The authors would like to thank Professors Martin Brokate, Pavel Krej\v{c}\'{i}, Vincenzo Recupero and anonymous referee for interesting remarks and references about the play operator. 


\end{document}